\documentclass[a4paper,11pt]{amsart}
\usepackage{mathrsfs}
\usepackage[all]{xy}
\usepackage{amsmath,amssymb,amscd,bbm,amsthm,mathrsfs}
\usepackage{graphicx}
\numberwithin{equation}{section}
\newtheorem{thm}{Theorem}[section]
\newtheorem{lem}{Lemma}[section]
\newtheorem{cor}{Corollary}[section]
\newtheorem{prop}{Proposition}[section]
\newtheorem{rem}{Remark}[section]

\makeatletter
\renewcommand{\@makefntext}[1]{\parindent 1em\noindent \hb@xt@1.8em{\hss\@makefnmark}#1}
\makeatother
\begin{document}

\title[Quaternionic M\"{o}bius invariant Laplacian and harmonic functions]{Quaternionic M\"{o}bius invariant Laplacian and quaternionic M\"{o}bius harmonic functions on the unit ball}
\author{Ruiwen Wang}
\thanks{School of Mathematical Sciences, Xiamen University, Fujian 361005, China, Email:rwwang@stu.xmu.edu.cn}
\begin{abstract}
We construct quaternionic M\"{o}bius ($\mathcal{QM}$ briefly) transformations on the quaternionic unit ball, which are used to define $\mathcal{QM}$-invariant Laplacian operator $\triangle$. A function annihilated by $\triangle$ is called $\mathcal{QM}$-harmonic. We prove that $\mathcal{QM}$-harmonic functions can be expanded in terms of quaternionic spherical harmonics multiplied by hypergeometric functions as radial parts. By establishing a Green formula associated to $\triangle$ and constructing the $\mathcal{QM}$-Poisson kernel, we solve the Dirichlet problem for $\mathcal{QM}$-invariant Laplace equation, which is degenerate elliptic. We also give a Fatou type theorem about non-tangential convergence of $\mathcal{QM}$-Poisson integrals. Compared to the real and complex cases, the main  difficulties come from the noncommutativity of the quaternionic algebra and the complexity of the quaternionic unitary group ${\rm Sp}(n){\rm Sp}(1)$ and its modules. However, they can be overcome by using the embedding of the quaternionic space to the complex matrix space and using more complicated algebraic tools.
\\ \par\

\end{abstract}

\maketitle
\section{Introduction}

On the unit ball of complex Euclidean space $\mathbb{C}^n$, one can use holomorphic M\"{o}bius automorphisms to define M\"{o}bius invariant Laplacian operator, which coincides with the Laplace-Beltrami operator on the complex hyperbolic manifold. The functions annihilated by this operator are called {\it M\"{o}bius harmonic}. For its real version, M\"{o}bius transformations on the unit ball of the Euclidean space are well-known, and there also exists an abundant theory of M\"{o}bius harmonic functions. The study of M\"{o}bius invariant subspaces of harmonic functions and operators on them is active in last 3 decades (cf.$\;$ e.g. \cite{Ahern,Liu1,Liu2,Liu3,Ou2,Rudin,Stoll,wulan,Xia} etc.). The purpose of this paper is to generalize fundamental facts of this theory to the quaternionic unit ball $B^{4n}:=\{\mathbf{q}\in \mathbb{H}^n;|\mathbf{q}|<1\}$.

 For $\mathbf{a}\in B^{4n}$, we denote by $\varphi_{\mathbf{a}}$ the following automorphism of $B^{4n}$:
\begin{equation}\begin{aligned}\label{def varphib}
\varphi_{\mathbf{a}}(\mathbf{q}):=(\mathbf{a}-\mathbf{A}\mathbf{q})(1-\mathbf{a}^*\mathbf{q})^{-1},
\end{aligned}\end{equation}
where $\mathbf{a},\mathbf{q}\in B^{4n}$ are $(n\times 1)$-quaternionic column vectors, and
\begin{equation}\begin{aligned}\label{hermitian}
\mathbf{A}=(1-s_{\mathbf{a}})|\mathbf{a}|^{-2}\mathbf{a}\mathbf{a}^*+s_{\mathbf{a}}I_n,\qquad s_{\mathbf{a}}=(1-|\mathbf{a}|^2)^{\frac12},
\end{aligned}\end{equation}
is a quaternionic $(n\times n)$-matrix, $\mathbf{a}^*=\bar{\mathbf{a}}^t$. Then a \emph{quaternionic M\"{o}bius transformation} ({\it $\mathcal{QM}$-transformation } briefly) is defined as the composition of a quaternionic unitary transformation in ${\rm Sp}(n){\rm Sp}(1)$ and $\varphi_{\mathbf{a}}$ for some $\mathbf{a}\in B^{4n}$. They constitute a group (cf. Theorem \ref{thm mobius}).

We can use the automorphism $\varphi_{\mathbf{a}}$ to define an operator on the unit ball:
\begin{equation}\begin{aligned}\label{def inv op}
(\triangle f)(\mathbf{a}):=\triangle_0(f\circ \varphi_{\mathbf{a}})(\mathbf{0}),
\end{aligned}\end{equation}
for $f\in C^2(B^{4n})$ and $\mathbf{a}\in B^{4n}$, where $\triangle_0$ is the standard Laplacian on $\mathbb{R}^{4n}$. $\triangle$ is proved to be invariant under $\mathcal{QM}$-transformations and called the \emph{quaternionic M\"{o}bius invariant Laplacian} ({\it $\mathcal{QM}$-invariant Laplacian} briefly). A function $f$ on $B^{4n}$ is called \emph{quaternionic M\"{o}bius harmonic} ({\it $\mathcal{QM}$-harmonic} briefly) if $\triangle f=0$.

One difficulty for generalizing the function theory on the real or complex unit ball to the quaternionic case is the noncommuntatity of the quaternionic algebra, which makes the construction of $\mathcal{QM}$-transformations and the proof of the explicit expression of $\mathcal{QM}$-invariant Laplacian more complicated. However, one can overcome this difficulty by embedding $\mathbb{H}^n$ to the complex matrix space and using complex vector fields introduced in \cite{wang5}, which are motivated by the embedding of the quaternionic algebra $\mathbb{H}$ into $\mathfrak{gl}(2,\mathbb{C})$:
\begin{equation}\begin{aligned}
q=x_0+x_1\mathbf{i}+x_2\mathbf{j}+x_3\mathbf{k}\mapsto \bigg(\begin{matrix}x_0+x_1\mathbf{i}&-x_2-x_3\mathbf{i}\\x_2-x_3\mathbf{i}&\;\;x_0-x_1\mathbf{i}\end{matrix}\bigg).
\end{aligned}\end{equation}
Namely, we use the embedding $\tau:\mathbb{H}^n\rightarrow\mathbb{C}^{2n\times 2}$ given by
\begin{equation}\begin{aligned}\label{zaa}
\tau(\mathbf{q})=(z_A^{A'})=\left(\begin{matrix}z_1^{0'}&z_1^{1'}\\z_2^{0'}&z_2^{1'}\\\vdots&\vdots\\z_{n+1}^{0'}&z_{n+1}^{1'}\\z_{n+2}^{0'}&z_{n+2}^{1'}\\\vdots&\vdots\end{matrix}\right):=\left(\begin{matrix}x_0+x_1\mathbf{i}&-x_2-x_3\mathbf{i}\\x_4+x_5\mathbf{i}&-x_6-x_7\mathbf{i}\\\vdots&\vdots\\x_2-x_3\mathbf{i}&\;\;x_0-x_1\mathbf{i}\\x_6-x_7\mathbf{i}&\;\;x_4-x_5\mathbf{i}\\\vdots&\vdots\end{matrix}\right),
\end{aligned}\end{equation}
if we write $q_{l-1}=x_{4l}+\mathbf{i}x_{4l+1}+\mathbf{j}x_{4l+2}+\mathbf{k}x_{4l+3}$, $l=1,2,\dots,n$, and complex vector fields
\begin{equation}\begin{aligned}\label{nabla}
\left(\nabla_{A}^{A'}\right)&=\frac{1}{2}\left(\begin{matrix}\nabla_1^{0'}& \nabla_{2}^{0'}&\\\vdots&\vdots\\ \nabla_{n+1}^{1'}&\nabla_{n+2}^{1'}\\\vdots&\vdots&\end{matrix}\right)
&:=\frac{1}{2}\left(\begin{matrix}\partial_{x_0}-\mathbf{i}\partial_{x_1}& -\partial_{x_2}+\mathbf{i}\partial_{x_3}\\\vdots&\vdots\\ \partial_{x_2}+\mathbf{i}\partial_{x_3}&\;\;\partial_{x_0}+\mathbf{i}\partial_{x_1}\\\vdots&\vdots&\end{matrix}\right),
\end{aligned}\end{equation}
where $A=1,2,\dots,2n$, $A'=0',1'$.
The quaternionic structure is encoded in these vector fields. We can write
\begin{equation}\begin{aligned}\label{def qj}
\mathbf{q}_j&=z_j^{0'}+\mathbf{j}z_{n+j}^{0'}=z_j^{0'}-z_{j}^{1'}\mathbf{j}.
\end{aligned}\end{equation}
The complex vector fields in \eqref{nabla} are so chosen to promise $\nabla_A^{A'}z_B^{B'}=\delta_{AB}\delta^{A'B'}$(cf. Lemma \ref{lem nabla}). In this notation, the\emph{ $\mathcal{QM}$-invariant Laplacian operator} $\triangle$ on $B^{4n}$ can be written explicitly as follows.
\begin{prop}\label{invariant laplace}
\begin{equation}\begin{aligned}\label{ex form}
\triangle=& 4(1-|\mathbf{q}|^2)\left( \frac14\triangle_0 + \mathfrak{D} -R^{0'}R^{1'}+ R^{0'}+R^{1'} \right),\\
\end{aligned}\end{equation}
where
\begin{equation}\begin{aligned}\label{def D R}
\mathfrak{D}=\sum_{A,B=1}^{2n}z_A^{1'} z_B^{0'}\nabla_A^{0'} \nabla_B^{1'},\qquad R^{A'}=\sum_{A=1}^{2n}z_A^{A'}\nabla_A^{A'},\qquad A'=0',1'.
\end{aligned}\end{equation}
\end{prop}

By the formula \eqref{ex form}, we see that $\triangle$ is degenerate elliptic operator. Another difficulty is that the quaternoinic unitary group ${\rm Sp}(n){\rm Sp}(1)$ has modules more complicated than that of $SO(n)$ and $U(n)$. In the complex case, we have the decomposition
\begin{equation*}\begin{aligned}
L^2(S^{2n-1})=\mathop{\bigoplus}_{\substack{a,b\in \mathbb{N}_0}}H_{a,b},
\end{aligned}\end{equation*}
 where $H_{a,b}$ consists of harmonic homogenous polynomials of total degree $a$ in $z_1,\dots,z_n$ and total degree $b$ in $\overline{z}_1,\dots,\overline{z}_n$. This is a crucial fact in the study of M\"{o}bius invariant Laplacian and M\"{o}bius harmonic functions on the complex unit ball \cite{Rudin}. In the quaternionic case, we have more complicated decomposition.
\\[8pt]
\vspace{\parskip}
\textbf{Theorem\; A} \label{highest weight vector}\cite[Proposition 2.1]{Co}
{\it There is a Hilbert space orthogonal direct sum decomposition
\begin{equation}\begin{aligned}\label{L2 decomp}
L^2(S^{4n-1})=\mathop{\bigoplus}_{\substack{a\geq 2b\\a,b\in \mathbb{N}_0}}\mathcal{H}_{a,b},
\end{aligned}\end{equation}
where $\mathcal{H}_{a,b}$ is an irreducible {\rm Sp}$(n)${\rm Sp}$(1)$-module consisting of harmonic polynomials (also called quaternionic spherical harmonics) on $\mathbb{R}^{4n}$, with the highest weight vector
\begin{equation}\begin{aligned}\label{pab}
p_{a,b}=\left(z_1^{0'}\right)^{a-b}\left(z_1^{0'}z_2^{1'}-z_1^{1'}z_2^{0'}\right)^{b}.
\end{aligned}\end{equation}
}
\\[8pt]
\vspace{\parskip}
$\qquad$It is proved in \cite{Co,Zhang} except for checking \eqref{pab} to be highest weight vector. We check it in Subsection 2.2.

\begin{thm}\label{no boundary decomp}
Suppose that $n\geq2$ and $u\in C^2(B^{4n})$ satisfying $\triangle u=0$. Then, there exists $h_{a,b}\in \mathcal{H}_{a,b}$ such that
\begin{equation}\begin{aligned}\label{no boundary decomp eq}
u(\mathbf{q})=\sum_{\substack{a\geq 2b\\a,b\in \mathbb{N}_0}}{}_2F_1\bigl(a,\; b-1,\; 2n+a+b;\; |\mathbf{q}|^2\bigr)h_{a,b}(\mathbf{q}),
\end{aligned}\end{equation}
for $\mathbf{q}\in B^{4n}$, and the series converges uniformly and absolutely on any compact subset of $B^{4n}$. Here, ${}_2F_1$ is the standard hypergeometric function.
\end{thm}

For $n=1$, $\mathcal{QM}$-transformations are the standard M\"{o}bius transformations on $\mathbb{R}^{4}$ and $\mathcal{QM}$-invariant Laplacian is the M\"{o}bius invariant Laplacian on the unit ball in $\mathbb{R}^{4}$ \cite{Stoll}. In this case, $L^2(S^{3})$ has the simple decomposition $L^2(S^{3})=\mathop{\bigoplus}\limits_{m\in \mathbb{N}_0}\mathcal{H}_{m}$, where $\mathcal{H}_{m}$ is the space of harmonic polynomials of degree $m$ \cite[Theorem 6.1.1]{Stoll}.

\begin{thm}\label{posson}
For $\varphi\in C(S^{4n-1})$, the Dirichlet problem
\begin{equation}\begin{aligned}\label{Dirichlet p}
\left\{
\begin{aligned}
 \triangle u&=0 {\rm\quad on}\quad B^{4n},\\
u&=\varphi {\rm \quad on}\quad S^{4n-1},
\end{aligned}
\right.
\end{aligned}\end{equation}
has a unique solution, which is given by the following $\mathcal{QM}$-Poisson integral
\begin{equation}\begin{aligned}\label{def poisson int}
P[\varphi](\mathbf{a}):=\int_{S^{4n-1}} P(\mathbf{a},\zeta)\varphi(\zeta)d\sigma(\zeta),
\end{aligned}\end{equation}
 where $\sigma$ is the ${\rm Sp}(n)$Sp$(1)$-invariant measure on $S^{4n-1}$ satisfying $\sigma(S^{4n-1})=1$, and
\begin{equation}\begin{aligned}\label{def poisson k}
P(\mathbf{a},\zeta):=\left(\frac{1-|\mathbf{a}|^2}{|1-\mathbf{a}^*\zeta|^2}\right)^{2n+1}
\end{aligned}\end{equation}
is the $\mathcal{QM}$-Poisson kernel for $\triangle$ on $B^{4n}\times S^{4n-1}$.
\end{thm}
We also prove the following Fatou type theorem about non-tangential convergence of $\mathcal{QM}$-Poisson integrals.
For $0<\alpha<\infty$, denote by $A_{\alpha}(\zeta)$ the \emph{non-tangential approach region} at $\zeta\in S^{4n-1}$
\begin{equation}\begin{aligned}\label{non tan}
A_{\alpha}(\zeta):=\{\mathbf{q}\in B^{4n};|\mathbf{q}-\zeta|<\alpha(1-|\mathbf{q}|)\}.
\end{aligned}\end{equation}

\begin{thm}\label{Fatou}
\;If $f\in L^1(S^{4n-1})$, then for $\alpha>1$ and almost all $\zeta$ on $S^{4n-1}$,
\begin{equation}\begin{aligned}\label{nontan lim}
\lim_{\substack{\mathbf{q}\rightarrow\zeta\\ \mathbf{q}\in A_{\alpha}(\zeta)}} P[f](\mathbf{q})=f(\zeta).
\end{aligned}\end{equation}

\end{thm}

  The paper is organized as follows. In Section 2, we construct $\mathcal{QM}$-transformations and show them constituting a group. In Section 3, we find the explicit expression of $\mathcal{QM}$-invariant Laplacian $\triangle$, and also that of the Casimir operators of ${\rm Sp}(n)$ and ${\rm Sp}(1)$. Then, in Section 4, by writing $\triangle$ in terms of Casimir operators and radial derivatives, we establish the expansion of $\mathcal{QM}$-harmonic functions in terms of quaternionic spherical harmonics multiplied by hypergeometric functions as radial parts. In Section 5, we prove a Green formula associated to $\triangle$ and construct the $\mathcal{QM}$-Poisson integral formula to solve the Dirichlet problem for $\mathcal{QM}$-invariant Laplace equation. In Section 6, we prove the non-tangential convergence of the $\mathcal{QM}$-Poisson integrals. In the appendix, we calculate Casimir operators of Sp$(n)$ and Sp$(1)$.

\section{$\mathcal{QM}$-transformations}

\subsection{The group  Sp$(n)$Sp$(1)$}
{\it Compact symplectic group} ${\rm Sp}(n)$ consists of all quaternionic $(n\times n)$-matrix such that $AA^*=I_n$, where $*$ is the quaternionic conjugate, and ${\rm Sp}(1)$ is the group of $\mathbb{R}$-linear transformations of $\mathbb{H}^n$ given by right multiplication by unit quaternions \cite{Co}. The ${\rm Sp}(n){\rm Sp}(1)$-action on $L^2(S^{4n-1})$ is defined as $[\pi(g)f](\mathbf{q})=f(U^*\mathbf{q}p)$, for $g=(U,p)\in {\rm Sp}(n){\rm Sp}(1)$.

The complexification of Lie algebra $\mathfrak{sp}(n)$ of the group ${\rm Sp}(n)$ is $\mathfrak{sp}(2n,\mathbb{C})$, which is a subalgebra of $\mathfrak{gl}(2n,\mathbb{C})$. Let $E_{i,j}$ be the complex $(n\times n)$-matrix with $1$ only at the $(i,j)$-entry and all other entries zero. Its Cartan subalgebra $\mathfrak{h}$ is spanned by the $(2n\times2n)$-matrices $H_i = \left(\begin{matrix}E_{i,i}&0\\0&-E_{i,i}\end{matrix}\right)$.

  We will correspondingly take as basis for the dual vector space $\mathfrak{h}^{*}$ the dual basis $L_j$ , where $\langle L_j , H_i\rangle = \delta_{i,j}$.
The roots of the Lie algebra $\mathfrak{sp}({2n},\mathbb{C})$ are the vectors $\pm L_i\pm L_j$, where the positive root are $L_i-L_j$, $L_i+L_j$, $2L_i$ ($1\leq i<j\leq n$) with eigenvectors \cite[Section 16.2]{Fulton}
\begin{equation}\begin{aligned}\label{spn basis}
X_{i,j}&=E_{i,j}-E_{n+j,n+i}=\bigg(\begin{matrix}E_{i,j} &0\\0&-E_{j,i}\end{matrix}\bigg),\qquad X^{-}_{i,j}=E_{j,i}-E_{n+i,n+j}=\bigg(\begin{matrix}E_{j,i} &0\\0&-E_{i,j}\end{matrix}\bigg),\\
Y_{i,j}&=E_{i,n+j}+E_{j,n+i}=\bigg(\begin{matrix}0&E_{i,j}+E_{j,i}\\0&0\end{matrix}\bigg),\qquad Y^{-}_{i,j}=E_{n+i,j}+E_{n+j,i}=\bigg(\begin{matrix}0&0\\E_{i,j}+E_{j,i}&0\end{matrix}\bigg),\\
U_i&=E_{i,n+i}=\bigg(\begin{matrix}0&E_{i,i}\\0&0\end{matrix}\bigg),\qquad\qquad\qquad\quad\qquad
U^{-}_i=E_{n+i,i}=\bigg(\begin{matrix}0&0\\E_{i,i}&0\end{matrix}\bigg).
\end{aligned}\end{equation}
Then, $\mathfrak{sp}(2n,\mathbb{C})=\mathfrak{g}_{+}\bigoplus \mathfrak{h}\bigoplus \mathfrak{g}_{-}$, with
\begin{equation}\begin{aligned}
\mathfrak{g}_{+}=span\left\{X_{i,j},Y_{i,j},U_i\right\},\qquad\mathfrak{g}_{-}=span\{X^{-}_{i,j},Y^{-}_{i,j},U^{-}_i\},
\end{aligned}\end{equation}
and $\mathfrak{sl}(2,\mathbb{C})=\mathfrak{g}_{+}'\bigoplus \mathfrak{h}'\bigoplus \mathfrak{g}_{-}'$ with $\mathfrak{g}_{+}'=\mathbb{C}U,$ $\mathfrak{g}_{-}'=\mathbb{C}U^{-}$, $\mathfrak{h}'=\mathbb{C}H$, where
\begin{equation}\begin{aligned}\label{sp1 basis}
H&=\bigg(\begin{matrix}1 &0\\0&-1\end{matrix}\bigg),\qquad U=\bigg(\begin{matrix}0&1\\0&0\end{matrix}\bigg),\qquad U^{-}=\bigg(\begin{matrix}0&0\\1&0\end{matrix}\bigg).
\end{aligned}\end{equation}

  A complex $\mathfrak{sp}(n)\mathfrak{sp}(1)$-module is a $\mathfrak{sp}(2n,\mathbb{C})\mathfrak{sl}(2,\mathbb{C})$-module.
A vector $X$ for a $\mathfrak{sp}(2n,\mathbb{C})\mathfrak{sl}(2,\mathbb{C})$-module is called the {\it highest weight vector} if $X$ is a eigenvector of $\mathfrak{h}$ and $\mathfrak{h}'$, and killed by $\mathfrak{g}_{+}$ and $\mathfrak{g}_{+}'$. A $\mathfrak{sp}(2n,\mathbb{C})\mathfrak{sl}(2,\mathbb{C})$-module $V$ is called the {\it highest weight module} if there exists a highest weight vector of $V$.

\subsection{$\mathcal{QM}$-transformations}

It is known \cite[(1.4)]{wang4} that for an invertible $(n+1)\times (n+1)-$matrix
\begin{equation}\begin{aligned}
g=\begin{pmatrix} \mathbf{a}_{1\times 1} & \mathbf{b}_{1\times n} \\ \mathbf{c}_{n\times 1} & \mathbf{d}_{n\times n} \end{pmatrix}\in GL(n+1,\mathbb{H}),
\end{aligned}\end{equation}
we can define a {\it quaternionic fractional linear transformation} $\varphi_{g}:\mathbb{H}^n \rightarrow \mathbb{H}^n\setminus \{\mathbf{a}+\mathbf{b}\mathbf{q}\neq 0\}$ by
\begin{equation}\begin{aligned}\label{def m g}
\mathbf{q}\mapsto \varphi_{g}(\mathbf{q}):=(\mathbf{c}+\mathbf{d}\mathbf{q})(\mathbf{a}+\mathbf{b}\mathbf{q})^{-1}.
\end{aligned}\end{equation}
Then, for $\mathbf{a}\in B^{4n}$ and $\mathbf{A}$ given by \eqref{hermitian}, let
\begin{equation}\begin{aligned}\label{def g b}
g_{\mathbf{a}}: =\begin{pmatrix} 1 & -\mathbf{a}^* \\ \mathbf{a} & -\mathbf{A} \end{pmatrix}\in GL(n+1,\mathbb{H}).
\end{aligned}\end{equation}
It is invertible because $\mathbf{A}$ in \eqref{hermitian} has two eigenspaces span$\{\mathbf{a}\}$ and $\mathbf{a}^{\bot}$ with eigenvalue 1 and $s$, respectively,
i.e.
\begin{equation}\begin{aligned}\label{eigen}
\mathbf{A}\mathbf{a}=\mathbf{a},\qquad\mathbf{A}\mathbf{v}=s_{\mathbf{a}}\mathbf{v},
\end{aligned}\end{equation}
for all $\mathbf{v}\in \mathbf{a}^{\bot}$. The quaternionic fractional linear transformation $\varphi_{g_{\mathbf{a}}}$ associated to $g_{\mathbf{a}}$ is exactly $\varphi_{\mathbf{a}}$ given by \eqref{def varphib}.

\begin{prop}
For any $\mathbf{a}\in B^{4n}$, $\varphi_{\mathbf{a}}$ is a diffeomorphism from $\overline{B^{4n}}$ to itself. It is an involution.

\end{prop}

\begin{proof}
$\varphi_{\mathbf{a}}$ is a diffeomorphism since for each $\widetilde{\mathbf{q}}\in \overline{B^{4n}}$, $\varphi_{\mathbf{a}}(\mathbf{q})=\widetilde{\mathbf{q}}$ has a unique solution $\mathbf{q}=(\mathbf{A}-\widetilde{\mathbf{q}}\mathbf{a}^*)^{-1}(\mathbf{a}-\widetilde{\mathbf{q}})$, where $\mathbf{A}-\widetilde{\mathbf{q}}\mathbf{a}^*$ is invertible by \eqref{eigen}.

  To see it is an involution, note that for any $\mathbf{q}\in B^{4n}$,
\begin{equation}\begin{aligned}\label{pra 1}
\varphi_{\mathbf{a}}(\varphi_{\mathbf{a}}(\mathbf{q}))=(\mathbf{a}-\mathbf{A}\varphi_{\mathbf{a}}(\mathbf{q}))(1-\mathbf{a}^*\varphi_{\mathbf{a}}(\mathbf{q}))^{-1},
\end{aligned}\end{equation}
by definition \eqref{def varphib},
where
\begin{equation}\begin{aligned}\label{involution 1}
(1-\mathbf{a}^*\varphi_{\mathbf{a}}(\mathbf{q}))^{-1}&=(1-\mathbf{a}^*(\mathbf{a}-\mathbf{A}\mathbf{q})(1-\mathbf{a}^*\mathbf{q})^{-1})^{-1}\\
&=(1-(|\mathbf{a}|^2-\mathbf{a}^*\mathbf{q})(1-\mathbf{a}^*\mathbf{q})^{-1})^{-1}=\frac{1}{s_{\mathbf{a}}^2}(1-\mathbf{a}^*\mathbf{q})
  \end{aligned}\end{equation}
by \eqref{eigen} and $\mathbf{A}=\mathbf{A}^*$. Now substitute \eqref{involution 1} into \eqref{pra 1} to get
\begin{equation}\begin{aligned}\label{involution 2}
\varphi_{\mathbf{a}}(\varphi_{\mathbf{a}}(\mathbf{q}))&=\frac{1}{s_{\mathbf{a}}^2}(\mathbf{a}-\mathbf{A}(\mathbf{a}-\mathbf{A}\mathbf{q})(1-\mathbf{a}^*\mathbf{q})^{-1})(1-\mathbf{a}^*\mathbf{q})\\
&=\frac{1}{s_{\mathbf{a}}^2}(-\mathbf{a}\mathbf{a}^*\mathbf{q}+\mathbf{A}^2\mathbf{q})=\mathbf{q},
\end{aligned}\end{equation}
by applying \eqref{eigen} to $\mathbf{q}$ decomposed in terms of $\mathbf{a}$ and $\mathbf{a}^{\bot}$. The Proposition is proved.
\end{proof}

A {\it quaternonic unitary transformation} is the linear transformation $(\mathbf{U},q_0):\mathbb{H}^n \rightarrow \mathbb{H}^n$, given by $\mathbf{q}\mapsto\mathbf{U}\mathbf{q}q_0^{*}$, for some $\mathbf{U}\in {\rm Sp}(n)$ and $q_0\in {\rm Sp}(1)$. In the complex case, the group of holomorphic M\"{o}bius transformations is exactly the group of biholomorphic transformations of the unit ball \cite{Rudin}. We can also show $\mathcal{QM}$-transformations constituting a group.

\begin{thm}\label{thm mobius}

(1)\;Given $\mathbf{a},\mathbf{b}\in B^{4n}$, there exists $(\mathbf{U}, q_0)\in {\rm Sp}(n){\rm Sp}(1)$ such that
\begin{equation}\begin{aligned}\label{m 1}
\varphi_{\mathbf{a}}\circ\varphi_{\mathbf{b}}=(\mathbf{U},q_0)\circ\varphi_{\mathbf{c}}\qquad {\rm with }\qquad\mathbf{c}=\varphi_{\mathbf{b}}(\mathbf{a}).
\end{aligned}\end{equation}

(2)\;For $\mathbf{a}\in B^{4n}$ and $(\mathbf{U}_1,q_1)\in {\rm Sp}(n){\rm Sp}(1)$, there exists $\mathbf{b}\in B^{4n}$ and $(\mathbf{U}_2,q_2)\in {\rm Sp}(n){\rm Sp}(1)$ such that
\begin{equation}\begin{aligned}\label{m 2}
\varphi_{\mathbf{a}}\circ (\mathbf{U}_1,q_1)=(\mathbf{U}_2,q_2)\circ\varphi_{\mathbf{b}}.
\end{aligned}\end{equation}
(3)\;$G=\{(\mathbf{U},q_0)\circ\varphi_{\mathbf{a}} ;(\mathbf{U},q_0)\in {\rm Sp}(n){\rm Sp}(1),\;\mathbf{a}\in B^{4n}\}$ is a group.
\end{thm}
$G$ is called the {\it group of $\mathcal{QM}$-transformations}. To prove Theorem \ref{thm mobius}, we need the following lemma.

\begin{lem}\cite[(2.5)]{wang4}\label{multip}
For $g,h\in  GL(n+1,\mathbb{H})$, $\varphi_{g}\circ\varphi_{h}=\varphi_{gh}.$
\end{lem}

\begin{proof}

For
$h=\begin{pmatrix} \mathbf{a}'_{1\times 1} & \mathbf{b}'_{1\times n} \\ \mathbf{c}'_{n\times 1} & \mathbf{d}'_{n\times n} \end{pmatrix}$, $\varphi_{h}(\mathbf{q})=(\mathbf{c}'+\mathbf{d}'\mathbf{q})(\mathbf{a}'+\mathbf{b}'\mathbf{q})^{-1}$ by definition \eqref{def m g}. Then
\begin{equation}\begin{aligned}\label{gh}
\varphi_{g}\circ\varphi_{h}(\mathbf{q})&=(\mathbf{c}+\mathbf{d}\varphi_{h}(\mathbf{q}))(\mathbf{a}+\mathbf{b}\varphi_{h}(\mathbf{q}))^{-1}\\
&=(\mathbf{c}\mathbf{a}'+\mathbf{d}\mathbf{c}'+(\mathbf{c}\mathbf{b}'+\mathbf{d}\mathbf{d}')\mathbf{q})(\mathbf{a}\mathbf{a}'+\mathbf{b}\mathbf{c}'+(\mathbf{a}\mathbf{b}'+\mathbf{b}\mathbf{d}')\mathbf{q})^{-1}.\\
\end{aligned}\end{equation}
But
\begin{equation*}\begin{aligned}
gh=\begin{pmatrix} \mathbf{a}\mathbf{a}'+\mathbf{b}\mathbf{c}' & \mathbf{a}\mathbf{b}'+\mathbf{b}\mathbf{d}' \\ \mathbf{c}\mathbf{a}'+\mathbf{d}\mathbf{c}'& \mathbf{c}\mathbf{b}'+\mathbf{d}\mathbf{d}' \end{pmatrix}
.\\
\end{aligned}\end{equation*}
So the right hand side of \eqref{gh} is $\varphi_{gh}(\mathbf{q})$. Lemma \ref{multip} is proved.
\end{proof}

\begin{proof}[Proof of Theorem \ref{thm mobius}]
(1)\;Since $\varphi_{\mathbf{c}}$ is an involution by Proposition \ref{multip}, it's equivalent to show
 \begin{equation}\begin{aligned}\label{abc}
\varphi_{\mathbf{a}}\circ\varphi_{\mathbf{b}}\circ\varphi_{\mathbf{c}}=(\mathbf{U},q_0).
\end{aligned}\end{equation}
Take $\mathbf{c}=\varphi_{\mathbf{b}}(\mathbf{a})$. Then,
\begin{equation}\begin{aligned}\label{c a def}
\mathbf{c}=(\mathbf{b}-\mathbf{B}\mathbf{\mathbf{a}})(1-\mathbf{b}^*\mathbf{a})^{-1}\qquad{\rm and}\qquad\mathbf{a}=\varphi_{\mathbf{b}}(\mathbf{c})=(\mathbf{b}-\mathbf{B}\mathbf{\mathbf{c}})(1-\mathbf{b}^*\mathbf{c})^{-1}.
\end{aligned}\end{equation}
where $\mathbf{B}$ is the matrix in \eqref{hermitian} corresponding to the point $\mathbf{b}\in B^{4n}$.
By taking quaternionic conjugate in the first identity in \eqref{c a def}, we get
\begin{equation}\begin{aligned}\label{def of c}
(1-\mathbf{a}^*\mathbf{b})\mathbf{c}^*&=\mathbf{b}^*-\mathbf{a}^*\mathbf{B}.
\end{aligned}\end{equation}

 To show \eqref{abc}, note that
\begin{equation}\begin{aligned}\label{gagbgc}
 g_{\mathbf{a}}\cdot  g_{\mathbf{b}}\cdot  g_{\mathbf{c}}&=\begin{pmatrix} 1 & -\mathbf{a}^* \\ \mathbf{a} & -\mathbf{A} \end{pmatrix}\begin{pmatrix} 1 & -\mathbf{b}^* \\ \mathbf{b} & -\mathbf{B} \end{pmatrix}\begin{pmatrix} 1 & -\mathbf{c}^* \\ \mathbf{c} & -\mathbf{C} \end{pmatrix}\\
&=\begin{pmatrix} 1-\mathbf{a}^*\mathbf{b} & -\mathbf{b}^*+\mathbf{a}^*\mathbf{B} \\ \mathbf{a}-\mathbf{A}\mathbf{b} & -\mathbf{a}\mathbf{b}^*+\mathbf{A}\mathbf{B} \end{pmatrix}\begin{pmatrix} 1 & -\mathbf{c}^* \\ \mathbf{c} & -\mathbf{C} \end{pmatrix}\\
&=\begin{pmatrix} 1-\mathbf{a}^*\mathbf{b}-(\mathbf{b}^*-\mathbf{a}^*\mathbf{B})\mathbf{c} & \mathfrak{e} \\ \mathfrak{d} & -(\mathbf{a}-\mathbf{A}\mathbf{b})\mathbf{c}^*+(\mathbf{a}\mathbf{b}^*-\mathbf{A}\mathbf{B})\mathbf{C} \end{pmatrix},
\end{aligned}\end{equation}
where
\begin{equation*}\begin{aligned}
\mathfrak{e}&=-(1-\mathbf{a}^*\mathbf{b})\mathbf{c}^*+(\mathbf{b}^*-\mathbf{a}^*\mathbf{B})\mathbf{C}=-(1-\mathbf{a}^*\mathbf{b})\mathbf{c}^*+(1-\mathbf{a}^*\mathbf{b})\mathbf{c}^*\mathbf{C}^*=\mathbf{0},
\end{aligned}\end{equation*}
by using \eqref{def of c} and $\mathbf{C}\mathbf{c}=\mathbf{c}$ in \eqref{eigen}, and
\begin{equation*}\begin{aligned}
\mathfrak{d}=\mathbf{a}(1-\mathbf{b}^*\mathbf{c})-\mathbf{A}(\mathbf{b}-\mathbf{B}\mathbf{c})=\mathbf{a}(1-\mathbf{b}^*\mathbf{c})-\mathbf{A}\mathbf{a}(1-\mathbf{b}^*\mathbf{c})=\mathbf{0},
\end{aligned}\end{equation*}
by using the second identity in \eqref{c a def}.
Thus,
\begin{equation*}\begin{aligned}
 g_{\mathbf{a}}\cdot  g_{\mathbf{b}}\cdot  g_{\mathbf{c}}:=\begin{pmatrix} q_0 & 0\\ 0 & \mathbf{V} \end{pmatrix},
\end{aligned}\end{equation*}
with $q_0$ and $\mathbf{V}$ given by the right hand side of \eqref{gagbgc}. So by Lemma \ref{multip}, we get
\begin{equation*}\begin{aligned}
\varphi_{\mathbf{a}}\circ\varphi_{\mathbf{b}}\circ\varphi_{\mathbf{c}}(\mathbf{q})= \mathbf{V}\mathbf{q}q_0^{-1}.
 \end{aligned}\end{equation*}

   On the other hand, for any $\mathbf{d}\in B^{4n}$, $\varphi_{\mathbf{d}}$ maps the sphere $S^{4n-1}$ to itself by Proposition \ref{multip}.
Consequently, $\varphi_{\mathbf{a}}\circ\varphi_{\mathbf{b}}\circ\varphi_{\mathbf{c}}$ also maps the sphere $S^{4n-1}$ to itself. So we must have
\begin{equation*}\begin{aligned}
|\mathbf{V}\mathbf{q}||q_0|^{-1}=1, \qquad{\rm if}\qquad |\mathbf{q}|=1.
\end{aligned}\end{equation*}
Thus, $\mathbf{V}|q_0|^{-1}\in {\rm SO}(4n)$ and
 $q_0|q_0|^{-1}\in {\rm Sp}(1)$. So $\mathbf{V}|q_0|^{-1}\in GL(n,\mathbb{H})\bigcap {\rm SO}(4n)={\rm Sp}(n)$, i.e. \eqref{abc} holds with $(\mathbf{V}|q_0|^{-1},q_0|q_0|^{-1})$.

(2)\; The action of both sides of \eqref{m 2} at $\mathbf{b}$ gives us $\varphi_{\mathbf{a}}(\mathbf{U}_1\mathbf{b}q_1^{*})=\mathbf{0}$. So we take
$\mathbf{b}=\mathbf{U}_1^{*}\mathbf{a}q_1$, i.e. $\mathbf{a}=\mathbf{U}_1\mathbf{b}q_1^{*}$. Then
\begin{equation}\begin{aligned}\label{matrix eq}
g_{\mathbf{a}}\cdot \begin{pmatrix} q_1 & 0\\ 0 & \mathbf{U}_1 \end{pmatrix}\cdot g_{\mathbf{b}} &=\begin{pmatrix} 1 & -\mathbf{a}^* \\ \mathbf{a} & -\mathbf{A} \end{pmatrix}
  \begin{pmatrix} q_1 & 0 \\ 0 & \mathbf{U}_1 \end{pmatrix}
 \begin{pmatrix} 1 & -\mathbf{b}^* \\ \mathbf{b} & -\mathbf{B} \end{pmatrix}\\
 &=\begin{pmatrix}q_ 1 -\mathbf{a}^*\mathbf{U}_1\mathbf{b}&\mathfrak{e}' \\ \mathfrak{d}' &-\mathbf{a}q_1\mathbf{b}^* +\mathbf{A}\mathbf{U}_1\mathbf{B} \end{pmatrix}=: \begin{pmatrix} (1-|\mathbf{a}|^2)q_1 & 0 \\ 0 & \widetilde{\mathbf{U}}_2 \end{pmatrix},
\end{aligned}\end{equation}
where
\begin{equation*}\begin{aligned}
\mathfrak{d}'&=\mathbf{a}q_1-\mathbf{A}\mathbf{U}_1\mathbf{b}=(\mathbf{a}-\mathbf{A}\mathbf{a})q_1=\mathbf{0},\\
\mathfrak{e}'&=-q_1\mathbf{b}^*+\mathbf{a}^*\mathbf{U}_1\mathbf{B}=-q_1\mathbf{b}^*+q_1\mathbf{b}^*\mathbf{U}_1^*\mathbf{U}_1\mathbf{B}=\mathbf{0},
\end{aligned}\end{equation*}
by definition of $\mathbf{b}$ and $\mathbf{B}=\mathbf{B}^{*}$ in \eqref{eigen}. By the same argument as in the proof of (1), we must have $\mathbf{U}_2=\widetilde{\mathbf{U}}_2(1-|\mathbf{a}|^2)^{-1}\in {\rm Sp}(n)$, $q_2=q_1\in {\rm Sp}(1)$. Now \eqref{matrix eq} implies $\varphi_{\mathbf{a}}\circ (\mathbf{U}_1,q_1)\circ\varphi_{\mathbf{b}}=(\mathbf{U}_2,q_2)$. \eqref{m 2} is proved.

(3)\; The composition is closed in G, since for any  $(\mathbf{U}_1,q_1)\circ \varphi_{\mathbf{a}}$ and $(\mathbf{U}_2,q_2)\circ \varphi_{\mathbf{b}}\in G$,
\begin{equation*}\begin{aligned}
(\mathbf{U}_1,q_1)\circ \varphi_{\mathbf{a}} \circ (\mathbf{U}_2,q_2)\circ \varphi_{\mathbf{b}}
&=(\mathbf{U}_1,q_1)\circ (\mathbf{U}_3,q_3) \circ \varphi_{\mathbf{c}}\circ\varphi_{\mathbf{b}}\\
&=(\mathbf{U}_1,q_1)\circ (\mathbf{U}_3,q_3)\circ (\mathbf{U}_4,q_4) \circ \varphi_{\mathbf{d}}\\
&=(\mathbf{U}_1\mathbf{U}_3\mathbf{U}_4,q_1q_3q_4)\circ\varphi_{\mathbf{d}}\in G,\\
\end{aligned}\end{equation*}
for some $\mathbf{c},\mathbf{d}\in B^{4n}$, $\mathbf{U}_3,\mathbf{U}_4\in {\rm Sp}(n)$ and $q_3, q_4\in {\rm Sp}(1)$ by (1) and (2). Moreover, we have
\begin{equation*}\begin{aligned}
((\mathbf{U}_1,q_1)\circ \varphi_{\mathbf{a}})^{-1}= \varphi_{\mathbf{a}}\circ (\mathbf{U}_1^{-1},q_1^{-1})=(\mathbf{U}_2,q_2)\circ\varphi_{\mathbf{b}}
\end{aligned}\end{equation*}
for some $\mathbf{b}\in B^{4n}$ and $(\mathbf{U}_2,q_2)\in {\rm Sp}(n){\rm Sp}(1)$ by (2). Thus $G $ is a group.
\end{proof}

\section{The $\mathcal{QM}$-invariant Laplace operator}
\subsection{Quaternionic M\"{o}bius invariance of the operator $\triangle$}
Let $M_{p\times m}(\mathbb{H})$ be the space of quaternionic $p\times m$-matrices. Let $\tau:M_{p\times m}(\mathbb{H})\rightarrow M_{2p\times 2m}(\mathbb{C})$ given by
\begin{equation}\begin{aligned}\label{tau ab}
\tau(A+\mathbf{j}B)=\left(\begin{matrix}A&-B\\\overline{B}&\;\;\overline{A}
\end{matrix}\right),
\end{aligned}\end{equation}
if we write a $(p\times m)$-matrix quaternionic matrix as $A+\mathbf{j}B$, where $A$ and $B$ are complex $p\times m$-matrices. In particular, in the case of $M_{1\times n}(\mathbb{H})=\mathbb{H}^n$, $\tau$ given by \eqref{tau ab} is exactly \eqref{zaa}.
\begin{lem}\cite[Proposition 2.3]{wang2}\label{tau MN}
(1) $\tau(\mathcal{M}\mathcal{N})=\tau(\mathcal{M})\tau(\mathcal{N})$ for a quaternionic $(p\times m)-$matrix $\mathcal{M}$ and a quaternionic $(m\times l)-$matrix $\mathcal{N}$. In particular, for $\mathbf{q}'=\mathcal{M}\mathbf{q}$ with $\mathbf{q},\mathbf{q}'\in \mathbb{H}^n$ and a quaternionic $(n\times n)-$matrix $\mathcal{M}$, we have
\begin{equation*}\begin{aligned}
\tau(\mathbf{q}')=\tau(\mathcal{M})\tau(\mathbf{q}),
\end{aligned}\end{equation*}
as complex $(2n\times 2)-$matrices.\\
(2)\; Let $I_n$ be the identity $(n\times n)$-matrix. $\mathcal{M}\in {\rm Sp}(n)$ if and only if
\begin{equation}\begin{aligned}\label{def J}
\tau(\mathcal{M})J\tau(\mathcal{M})^t=J,\qquad{\rm where}\qquad J=\left(J_A^B\right):=\left(\begin{matrix}0&I_n\\-I_n&0\end{matrix}\right),
\end{aligned}\end{equation}
where $A$ and $B$ are row and column indices, respectively.\\
(3)\;$\tau(\mathcal{M}^*)=\overline{\tau(\mathcal{M})}^t$.
\end{lem}
We adopt the notations of indices $i,j,l\dots\in\{1,2,\dots,n\}$ and $A,B,\dots\in\{1,2,\dots,2n\}$, $A',B',\dots\in \{0',1'\}$. The summation of $A,B,\dots$ are always taken over $1,2,\dots,2n$.
\begin{lem}\label{lem nabla}\cite[Lemma 3.1]{wang3}
For $\nabla_A^{A'}$ and $z_{B}^{B'}$ given by \eqref{nabla}, we have
\begin{equation}\begin{aligned}\label{delta}
\nabla_{A}^{A'}z_{B}^{B'}=\delta_{AB}\delta^{A'B'}.
\end{aligned}\end{equation}

\end{lem}
By definition $z_A^{A'}$ in \eqref{zaa} and $J$ in \eqref{def J}, we can check directly
\begin{equation}\begin{aligned}\label{z jab}
z_A^{1'}=\sum_B-J_A^B\overline{z_B^{0'}},\qquad\sum_{A,B}z_A^{A'}J_A^Bz_B^{B'}=\varepsilon^{A'B'}|\mathbf{q}|^2,
\end{aligned}\end{equation}
where $\left(\varepsilon^{A'B'}\right)=\left(\begin{matrix}0&1\\-1&0\end{matrix}\right)$ and
\begin{equation}\begin{aligned}\label{z jab2}
\triangle_0=4\sum_{A,B}\nabla_A^{0'}\nabla_{B}^{1'}J_A^B,\qquad  \sum_{A}z_A^{A'}\nabla_A^{A'}|\mathbf{q}|^2=|\mathbf{q}|^2.
\end{aligned}\end{equation}
where $A'=0',1'$. Namely, the Euclidean norm equals to symplectic product of $\left(z_{A}^{0'}\right)$ and $\left(z_{A}^{1'}\right)$ and $\triangle_0$ is symplect product of $\left(\nabla_{A}^{0'}\right)$ and $\left(\nabla_{A}^{1'}\right)$.

\begin{prop}\label{m invariant}
$\triangle$ is $\mathcal{QM}$-invariant, i.e.
\begin{equation*}\begin{aligned}
\triangle(f\circ\psi)(\mathbf{a})=\triangle f(\psi(\mathbf{a})),
\end{aligned}\end{equation*}
for any $\mathcal{QM}$-transformation $\psi$, $f\in C^2(\Omega)$ and $\mathbf{a}\in B^{4n}$.
\end{prop}

\begin{proof}
For $\psi\in G$, let $\mathbf{b}=\psi(\mathbf{a})$. Then, $\varphi_{\mathbf{b}}\circ\psi\circ\varphi_{\mathbf{a}}$ is a $\mathcal{QM}$-transformation fixing $\mathbf{0}$, and there exists $(\mathbf{U}',q_0')\in {\rm Sp}(n){\rm Sp}(1)$ and $\mathbf{c}\in B^{4n}$ such that
\begin{equation*}\begin{aligned}
\varphi_{\mathbf{b}}\circ\psi\circ\varphi_{\mathbf{a}}&=\varphi_{\mathbf{c}}\circ (\mathbf{U}',q_0'),
\end{aligned}\end{equation*}
by Theorem \ref{thm mobius} (1)-(2). Then,
\begin{equation*}\begin{aligned}
\mathbf{0}=\varphi_{\mathbf{c}}\circ (\mathbf{U}',q_0')(\mathbf{0})=\varphi_{\mathbf{c}}(\mathbf{0})=\mathbf{c}.
\end{aligned}\end{equation*}
Thus, $\varphi_{\mathbf{c}}=id$ and $\varphi_{\mathbf{b}}\circ\psi\circ\varphi_{\mathbf{a}}= (\mathbf{U}',q_0')\in {\rm Sp}(n){\rm Sp}(1)$.
Hence,
\begin{equation*}\begin{aligned}
\triangle(f\circ\psi)(\mathbf{a})=&\triangle_0 (f\circ\psi\circ\varphi_{\mathbf{a}})(\mathbf{0})=\triangle_0 (f\circ\varphi_{\mathbf{b}}\circ (\mathbf{U}',q_0'))(\mathbf{0})\\
=&\triangle_0 (f\circ\varphi_{\mathbf{b}})(\mathbf{0})=\triangle f(\mathbf{b})=\triangle f(\psi(\mathbf{a})),
\end{aligned}\end{equation*}
 by the invariance of the standard Laplacian $\triangle_0$ under the group Sp$(n)$Sp$(1)\subset{\rm SO}(4n)$.
\end{proof}
\begin{rem}
In \cite{wang4}, it is proved that $k$-Cauchy-Fueter operator is quaternionic projective invariant, which implies it is $\mathcal{QM}$-invariant on $B^{4n}$. The functions annihilated by this operator is $k$-regular functions. These functions on quaternionic hyperbolic space were investigated in \cite{Chang3,CHANG,hr}. So we can define them in the $\mathcal{QM}$-invariant way as in \eqref{def inv op}. It is interesting to use group representation to analyze them as here for $\triangle$.
\end{rem}

\subsection{The explicit expression of $\mathcal{QM}$-invariant Laplacian}
To find the expression of $\triangle$ in Proposition \ref{invariant laplace}, we need to calculate $\triangle_0(f\circ \varphi_{\mathbf{a}})(\mathbf{0})$. In the complex case, the calculation is easy since $\varphi_{\mathbf{a}}$ is holomorphic. Here we need to know first order derivatives of $\nabla_{A}^{A'}\varphi_{\mathbf{a}}(\mathbf{0})$ and $\triangle_0\varphi_{\mathbf{a}}(\mathbf{0})$.
\begin{lem}\label{1st differ}
Let $w(\mathbf{q}):=\tau(\varphi_{\mathbf{a}}(\mathbf{q}))$ as a function of $\mathbf{q}$. Then
\begin{equation}\begin{aligned}\label{rewritten eq 1st de AneqB}
\nabla_{B}^{B'} w(\mathbf{0})=&\left(-s+\frac{s}{1+s}\tau(\mathbf{a}\mathbf{a}^*)\right)E_B^{B'},\\
\triangle_0 w(\mathbf{0})&=4(1-|\mathbf{a}|^2)\tau(\mathbf{a}),
\end{aligned}\end{equation}
where $E_B^{B'}$ be the $(2n\times 2)$-matrix with $1$ only in $(B,B')$-entry and $0$ all other entries.
\end{lem}
\begin{proof}

Expand $\varphi_{\mathbf{a}}(\mathbf{q})$ in \eqref{def varphib} at $\mathbf{q}=0$ to get
\begin{equation}\begin{aligned}\label{taylor}
\varphi_{\mathbf{a}}(\mathbf{q})
=&\bigg(\mathbf{a}-(1-s)|\mathbf{a}|^{-2}\mathbf{a}\mathbf{a}^*\mathbf{q}-s\mathbf{q}\bigg)\bigg(1+\mathbf{a}^*\mathbf{q}+\left(\mathbf{a}^*\mathbf{q}\right)^2+O(|\mathbf{q}|^3)\bigg).
\end{aligned}\end{equation}
Taking the mapping $\tau$, we get
\begin{equation}\begin{aligned}
w=\left(a_A^{A'}\right)+\Sigma_1+\Sigma_2+O(|z_A^{A'}|^2),
\end{aligned}\end{equation}
where
\begin{equation}\begin{aligned}
\Sigma_1:=&-s\left(z^{A'}_A\right)+\frac{s}{1+s}\tau(\mathbf{a})\tau(\mathbf{a}^*)\left(z^{A'}_A\right),\\
\Sigma_2:=&\frac{s}{1+s}\tau(\mathbf{a})\tau(\mathbf{a}^*)\left(z^{A'}_A\right)\tau(\mathbf{a}^*)\left(z^{A'}_A\right)-s\left(z^{A'}_A\right)\tau(\mathbf{a}^*)\left(z^{A'}_A\right).
\end{aligned}\end{equation}
Then, by $\nabla_B^{B'}\left(z_A^{A'}\right)=E_B^{B'}$, we get first identity in \eqref{rewritten eq 1st de AneqB}, and
\begin{equation}\begin{aligned}
\nabla_B^{B'}\Sigma_2=&\bigg(\frac{s}{1+s}\tau(\mathbf{a})\tau(\mathbf{a}^*)-sI_{2n}\bigg)\bigg(E_B^{B'}\tau(\mathbf{a}^*)\left(z_A^{A'}\right)+\left(z_A^{A'}\right)\tau(\mathbf{a}^*)E_B^{B'}\bigg).\\
\end{aligned}\end{equation}
Then, we get
\begin{equation}\begin{aligned}
\frac14\triangle_0 w(\mathbf{0})=&\sum_{A,B}J_A^B\nabla_A^{0'}\nabla_B^{1'}\Sigma_2\\
=&\sum_{A,B}\bigg(\frac{s}{1+s}\tau(\mathbf{a})\tau(\mathbf{a}^*)-sI_{2n}\bigg)\left(E_B^{1'}\tau(\mathbf{a}^*)E_A^{0'}+E_A^{0'}\tau(\mathbf{a}^*)E_B^{1'}\right)J_A^B\\
=&\sum_{A,B}\left(\frac{s}{1+s}\left(\sum_{A'}a_C^{A'}\overline{a_D^{A'}}\right)-sI_{2n}\right)\bigg(\overline{a_A^{1'}}E_B^{0'}+\overline{a_B^{0'}}E_A^{1'}\bigg)J_A^B\\
=&\frac{s}{1+s}\left(\sum_{A,B,A'}a_C^{A'}\overline{a_B^{A'}}\overline{a_A^{1'}}J_A^B,\mathbf{0}\right)+\frac{s}{1+s}\left(\mathbf{0},\sum_{A,B,A'}a_C^{A'}\overline{a_A^{A'}}\overline{a_B^{0'}}J_A^B\right)\\
&-\sum_{A,B}s\left(\overline{a_A^{1'}}J_A^BE_B^{0'}+\overline{a_B^{0'}}J_A^BE_A^{1'}\right)\\
=&\frac{-s}{1+s}|\mathbf{a}|^2\tau(\mathbf{a})+s\tau(\mathbf{a}),
\end{aligned}\end{equation}
by using \eqref{z jab} repeatedly in the last identity. The lemma is proved.
\end{proof}

\begin{lem}\label{chain rule}
For a smooth function $f$ on $\mathbb{R}^{4n}$ and a smooth transformation $\psi$, we have
\begin{equation}\begin{aligned}
\nabla_{A}^{A'}(f\circ \psi)(\mathbf{q})&=\sum_{B,B'}\left(\nabla_{B}^{B'}f\right)\left(\psi(\mathbf{q})\right)\cdot\nabla_{A}^{A'}w_{B}^{B'}(\mathbf{q}),\\
\end{aligned}\end{equation}
where $\tau(\psi(\mathbf{q}))=(w_{B}^{B'}(\mathbf{q}))$.
\end{lem}
\begin{proof}
It is sufficient to check the formula for $f$ and $\psi$ real analytic. Note that a real analytic function $F(x_0,\dots,x_{4n-1})$ on $\mathbb{R}^{4n}$ can be extended to a holomorphic function on $\mathbb{C}^{4n}$ by
\begin{equation}\begin{aligned}\label{holo extend}
\underline{F}(z)=F\left(\dots,\frac{z_{l}^{0'}+z_{n+l}^{1'}}{2},\frac{z_{l}^{0'}-z_{n+l}^{1'}}{2\mathbf{i}},\frac{z_{n+l}^{0'}-z_{l}^{1'}}{2},\frac{z_{n+l}^{0'}+z_{l}^{1'}}{-2\mathbf{i}},\dots\right),
\end{aligned}\end{equation}
which satisfies $F(\mathbf{q})=\underline{F}(\tau(\mathbf{q}))$ by the embedding $\tau:\mathbb{H}^n\rightarrow \mathbb{C}^{4n}$ in \eqref{zaa}.
It is easy to see that \cite[(3.9)]{wang4}
\begin{equation}\begin{aligned}\label{leb}
\nabla_A^{A'}[\underline{F}(\tau(\mathbf{q}))]=\partial_A^{A'}\underline{F}\left(\tau(\mathbf{q})\right),
\end{aligned}\end{equation}
where $\partial_A^{A'}=\frac{\partial}{\partial z_A^{A'}}$ is holomorphic derivative with respect to $z_A^{A'}$. Similarly, $\psi$ can be extended to local holomorphic transformations $\Psi=(\Psi_B^{B'}):\mathbb{C}^{4n}\rightarrow\mathbb{C}^{4n}$. Then $f\circ \psi(\mathbf{q})=\underline{F}\circ\underline{\Psi}(\tau(\mathbf{q}))$.

By the chain rule for holomorphic functions under holomorphic transformation, we have
\begin{equation*}\begin{aligned}
\partial_{A}^{A'}(\underline{F}\circ \Psi)(z)=\sum_{B,B'}(\partial_{B}^{B'}\underline{F})\circ\Psi(z)\cdot\partial_{A}^{A'}\Psi_{B}^{B'}(z).
\end{aligned}\end{equation*}
If restricting to the total real subspace $\tau(\mathbb{H}^n)\in \mathbb{C}^{4n}$ and using \eqref{leb}, we get the result.
\end{proof}

\begin{proof}[Proof of Proposition \ref{invariant laplace}]
Note that $\nabla_{B}^{B'}w_{A}^{A'}(\mathbf{0})=0$ for $A'\neq B'$, because
\begin{equation}\begin{aligned}\label{1st differ w}
\nabla_{B}^{B'} w_{A}^{A'}(\mathbf{0})=&-s\delta_{BA}^{B'A'}+\frac{s}{1+s}\sum_{D'}a_A^{D'}\overline{a_B^{D'}}\delta^{B'A'},\\
\end{aligned}\end{equation}
by the first identity of \eqref{rewritten eq 1st de AneqB} and $(C,D)$-entry of $\tau(\mathbf{a})\tau(\mathbf{a}^*)$ is $\sum_{D'}a_C^{D'}\overline{a_D^{D'}}$.

By using the chain rule in Lemma \ref{chain rule} repeatedly, we get
\begin{equation}\begin{aligned}\label{**}
\triangle(f\circ\varphi_{\mathbf{a}})(\mathbf{0})=&4\sum_{A,B}J_A^B\nabla_{A}^{0'}\nabla_{B}^{1'}(f\circ \varphi_{\mathbf{a}})(\mathbf{0})\\
=&4\sum_{A,B,C}J_A^B\nabla_{A}^{0'}\bigg[\sum_C(\nabla_{C}^{0'}f)\circ\varphi_{\mathbf{a}}\cdot\nabla_{B}^{1'} w_{C}^{0'}+(\nabla_{C}^{1'} f)\circ\varphi_{\mathbf{a}}\cdot\nabla_{B}^{1'}w_{C}^{1'}\bigg](\mathbf{0})\\
=&\sum_{C}\bigg(\nabla_{C}^{0'}f(\mathbf{a})\cdot\triangle_0 w_{C}^{0'}(\mathbf{0})+\nabla_{C}^{1'}f(\mathbf{a})\cdot\triangle_0 w_{C}^{1'}(\mathbf{0})\bigg)\\
&+4\sum_{A,B,C,D}J_A^B\nabla_{D}^{0'}\nabla_{C}^{1'}f(\mathbf{a})\cdot\nabla_{A}^{0'} w_{D}^{0'}(\mathbf{0})\cdot\nabla_{B}^{1'}w_{C}^{1'}(\mathbf{0})=:\widehat{\Sigma}_1+\widehat{\Sigma}_2.
\end{aligned}\end{equation}
Now by the second identity in Lemma \ref{1st differ}, we get $$\widehat{\Sigma}_1=4(1-|\mathbf{a}|^2)(R^{0'}+R^{1'})f(\mathbf{a}).$$
By \eqref{1st differ w}, we get
\begin{equation*}\begin{aligned}
&\sum_{A,B}J_A^B\nabla_{A}^{0'} w_{D}^{0'}(0)\cdot\nabla_{B}^{1'}w_{C}^{1'}(0)\\
=&\sum_{A,B}J_A^B\bigg(-s\delta_{AD}+\frac{s}{1+s}\sum_{A'}a_D^{A'}\overline{a_A^{A'}}\bigg)
\bigg(-s\delta_{BC}+\frac{s}{1+s}\sum_{B'}a_C^{B'}\overline{a_B^{B'}}\bigg) \\
=&s^2J_D^C-\frac{s^2}{1+s}\bigg(\sum_{B,B'}J_D^Ba_C^{B'}\overline{a_B^{B'}}+\sum_{A,A'}J_A^Ca_D^{A'}\overline{a_A^{A'}}\bigg)+\bigg(\frac{s}{1+s}\bigg)^2\sum_{A,B,A',B'}a_D^{A'}a_C^{B'}\overline{a_A^{A'}}\overline{a_B^{B'}}J_A^B\\
=&s^2J_D^C+\bigg(\left(\frac{s}{1+s}\right)^2|\mathbf{a}|^2-\frac{2s^2}{1+s}\bigg)a_{[D}^{0'} a_{C]}^{1'}\\
=&(1-|\mathbf{a}|^2)\left(J_D^C-a_{[D}^{0'} a_{C]}^{1'}\right),
\end{aligned}\end{equation*}
by using \eqref{z jab} repeatedly. Here we use the notation $a_{[A}^{A'}a_{B]}^{B'}:=a_{A}^{A'}a_{B}^{B'}-a_{B}^{A'}a_{A}^{B'}$.
Thus, we get
\begin{equation*}\begin{aligned}
\widehat{\Sigma}_2
=4(1-|\mathbf{a}|^2)\sum_{C,D=1}^{2n}\left(J_D^C-a_{[D}^{0'}a_{C]}^{1'}\right)\nabla_{D}^{0'}\nabla_{C}^{1'} f=4(1-|\mathbf{a}|^2)\left( \frac14\triangle_0 + \mathfrak{D} -R^{0'}R^{1'} \right)f.\\
\end{aligned}\end{equation*}
The sum of $\widehat{\Sigma}_1$ and $\widehat{\Sigma}_2$ gives us \eqref{ex form}. The proposition is solved.
\end{proof}

\begin{rem}
If we take $z_k=z_k^{0'}$ and $\bar{z}_k=z_k^{1'}$ for $k=1,2,\dots,2n$, \eqref{ex form} can be rewritten as
\begin{equation*}\begin{aligned}
\triangle=&4(1-|z|^2)\bigg[\sum_{i,j=1}^{n}\bigg((\delta_{ij}-z_i\bar{z}_j-\bar{z}_{n+i}z_{n+j})\partial_{z_i}\partial_{\bar{z}_j}+(\bar{z}_iz_{n+j}-z_{n+i}\bar{z}_{j})\partial_{z_{n+i}}\partial_{\bar{z}_j}\\
&+(\bar{z}_{n+i}z_j-z_i\bar{z}_{n+j})\partial_{z_i}\partial_{\bar{z}_{n+j}}+(\delta_{ij}-\bar{z}_iz_j-z_{n+i}\bar{z}_{n+j})\partial_{z_{n+i}}\partial_{\bar{z}_{n+j}}\bigg)+\bigg(\sum_{i=1}^{2n}z_i\partial_{z_i}+\bar{z}_i\partial_{\bar{z}_i}\bigg)\bigg],
\end{aligned}\end{equation*}
which concides with the expression of Beltrami-Laplacian with respect to the quaternionic hyperbolic metric in \cite{Lu}.

\end{rem}

\subsection{Casimir operators of {\rm Sp}($n$) and {\rm Sp}($1$)}
To prove Theorem \ref{no boundary decomp}, we need to know the action of $\triangle$ on the Sp$(n)$Sp$(1)$-modules $\mathcal{H}_{a,b}$. For this purpose, we express $\triangle$ as the sum of Casimir operators and radial operators. A complex representation of a Lie algebra $\mathfrak{g}$ is a homomorphism $\phi:\mathfrak{g}\rightarrow \mathfrak{gl}(V)$, where $V$ is a vector space over $\mathbb{C}$.
Let $\pi$ be the representation of $\mathfrak{sp}(2n,\mathbb{C})$ and $\mathfrak{sl}(2,\mathbb{C})$ on $\mathbb{C}^{2n\times 2}$ as holomorphic vector fields given by
\begin{equation}\begin{aligned}\label{representation def}
\pi(M)f(\mathbf{z})&=\left.\frac{d}{ds}f\left(e^{s\overline{M}^t}\mathbf{z} \right)\right|_{s=0}, \qquad \pi(N)f(\mathbf{z})&=\left.\frac{d}{ds}f\left(\mathbf{z}e^{sN} \right)\right|_{s=0},
\end{aligned}\end{equation}
where $M=(M_{AB})\in \mathfrak{sp}(2n,\mathbb{C})$ and $N=(N_{A'B'})\in \mathfrak{sl}(2,\mathbb{C})$, respectively.

 Note that a polynomial on $\mathbb{R}^{4n}$ can be extended to a holomorphic polynomial on $\mathbb{C}^{2n\times 2}$ by \eqref{holo extend}, and conversely, a holomorphic polynomial restricted to $\tau(\mathbb{H})$ gives us a polynomial on $\mathbb{R}^{4n}$. Moreover, ${\rm Sp}(n){\rm Sp}(1)$ preserve $\tau(\mathbb{H})$ by Lemma \ref{tau MN} (2). For $X\in\mathfrak{sp}(n)$ and $X'\in\mathfrak{sp}(1)$, $\tau(X^*)=\overline{\tau(X)}^t$ by Lemma \ref{tau MN} (3) and \eqref{z jab}. Thus, \eqref{representation def} give us the vector fields on $\mathbb{R}^{4n}$ representing $\mathfrak{sp}(n)$ and $\mathfrak{sp}(1):$
\begin{equation}\begin{aligned}\label{representation defd}
\pi(X,X')f(\mathbf{z})&=\left.\frac{d}{ds}f\left(e^{s\tau(X^{*})}\mathbf{z}e^{sX'} \right)\right|_{s=0}.
\end{aligned}\end{equation}

\begin{lem}\label{lemma lie}
For $M=(M_{AB})\in \mathfrak{sp}(2n,\mathbb{C})$ and $N=(N_{A'B'})\in \mathfrak{sl}(2,\mathbb{C})$,
\begin{equation}\begin{aligned}\label{lie}
\pi(M)=\sum_{A,B,A'}\overline{M}_{BA}z_B^{A'}\nabla_A^{A'},\qquad\pi(N)=\sum_{A,A',B'}z_A^{B'}N_{B'A'}\nabla_A^{A'}.
\end{aligned}\end{equation}
\end{lem}

\begin{proof}
For a holomorphic polynomial $f$ on $\mathbb{C}^{2n}$,
\begin{equation*}\begin{aligned}
\pi(M)f(\mathbf{z})&=\left.\frac{d}{ds}\right|_{s=0}f(\mathbf{z}+s\overline{M}^t\mathbf{z}+O(s^2))=\sum_{A,B,A'}(\overline{M}^t)_{AB}z_B^{A'}\partial_A^{A'}f(\mathbf{z}),\\
\pi(N)f(\mathbf{z})&=\left.\frac{d}{ds}\right|_{s=0}f(\mathbf{z}+s\mathbf{z}N+O(s^2))=\sum_{A,A',B'}z_A^{B'}N_{B'A'}\partial_A^{A'}f(\mathbf{z}),\\
\end{aligned}\end{equation*}
by definition \eqref{representation def}. The lemma is proved.
\end{proof}

\begin{proof}[Proof of Theorem A]
See \cite[Proposition 2.1]{Co} for the decomposition \eqref{L2 decomp}.

  By Lemma \ref{lemma lie}, we see that vector fields for basis of $\mathfrak{sp}(2n,\mathbb{C})$ in \eqref{spn basis} are given by
\begin{equation}\begin{aligned}\label{Xij Yij Ui}
\pi(H_{i,i})&= \sum_{A'} (z_{i}^{A'} \nabla_{i}^{A'} - z_{n+i}^{A'} \nabla_{n+i}^{A'}),\\
\pi(X_{i,j})&=\sum_{A'} (z_{i}^{A'} \nabla_{j}^{A'} - z_{n+j}^{A'} \nabla_{n+i}^{A'}),\qquad\pi(X_{i,j}^-)= \sum_{A'} (z_{j}^{A'} \nabla_{i}^{A'} - z_{n+i}^{A'} \nabla_{n+j}^{A'}),\\
\pi(Y_{i,j})&= \sum_{A'} (z_{i}^{A'} \nabla_{n+j}^{A'}+z_{j}^{A'} \nabla_{n+i}^{A'}),\qquad\pi(Y_{i,j}^{-})= \sum_{A'} (z_{n+i}^{A'} \nabla_{j}^{A'}+z_{n+j}^{A'} \nabla_{i}^{A'}),\\
\pi(U_i)&=\sum_{A'}  z_{i}^{A'} \nabla_{n+i}^{A'},\qquad\qquad\qquad\quad\pi(U_i^-) = \sum_{A'} z_{n+i}^{A'} \nabla_{i}^{A'},
\end{aligned}\end{equation}
where $1\leq i<j\leq n$, while the vector fields for basis of $\mathfrak{sp}(2,\mathbb{C})$ in \eqref{sp1 basis} are given by
\begin{equation}\begin{aligned}\label{u H}
\pi(U) =&\sum_{A} z_{A}^{0'} \nabla_{A}^{1'},\qquad\pi(U^{-}) = \sum_{B}  z_{B}^{1'} \nabla_{B}^{0'},\\
\pi(H) =& \sum_{A}  z_{A}^{0'} \nabla_{A}^{0'} - \sum_{A}z_{A}^{1'} \nabla_{A}^{1'}:=R^{0'}-R^{1'}.
\end{aligned}\end{equation}

  Since $p_{a,b}$ in \eqref{pab} is independent of $z_A^{A'}$ for $A\geq 3$, we see that for $2\leq i<j$, $1\leq i'$,
\begin{equation}\begin{aligned}\label{pi pab}
\pi(X_{i,j})p_{a,b}=\pi(Y_{i,j})p_{a,b},\quad\pi(U_{i'})p_{a,b}=0,
\end{aligned}\end{equation}
by the Leibniz law. For $i=1$, $j=2$, we have $\pi(Y_{1,2})p_{a,b}=0$ by the same reason, and
\begin{equation}\begin{aligned}\label{3.24}
\pi(X_{1,2})p_{a,b}
=(z_1^{0'})^{a-b}(z_1^{0'}z_2^{1'}-z_1^{1'}z_2^{0'})^{b-1}(z_1^{0'}z_1^{1'}-z_1^{1'}z_1^{0'})=0.
\end{aligned}\end{equation}
Similarly,
\begin{equation}\begin{aligned}
R^{0'}p_{a,b}&=ap_{a,b}\qquad R^{1'}p_{a,b}&=bp_{a,b},\qquad\pi(H_{i,i})p_{a,b}=0,\quad for \quad i\geq 3,
\end{aligned}\end{equation}
and
\begin{equation}\begin{aligned}\label{piH}
\pi(U)p_{a,b}=(z_1^{0'}\nabla_1^{1'}+z_2^{0'}\nabla_2^{1'})p_{a,b}=0,\qquad\pi(H)p_{a,b}=(R^{0'}-R^{1'})p_{a,b}=(a-b)p_{a,b}.
\end{aligned}\end{equation}
We also have
\begin{equation}\begin{aligned}\label{pi H11H22}
\pi(H_{1,1})p_{a,b}=&(z_{1}^{0'} \nabla_{1}^{0'} + z_{1}^{1'} \nabla_{1}^{1'})\left[(z_1^{0'})^{a-b}(-z_1^{0'}z_2^{1'}+z_2^{0'}z_1^{1'})^{b}\right]=ap_{a,b},\\
\pi(H_{2,2})p_{a,b}=&(z_1^{0'})^{a-b}(z_2^{0'} \nabla_{2}^{0'}+z_2^{1'} \nabla_{2}^{1'})(-z_1^{0'}z_2^{1'}+z_2^{0'}z_1^{1'})^{b}=bp_{a,b}.
\end{aligned}\end{equation}
Thus, $p_{a,b}$ is annihilated by $\mathfrak{g}_{+}$ and $\mathfrak{g}_{+}'$ and is the eigenvector of $\mathfrak{h}$ and $\mathfrak{h}'$. So it is a highest weight vector.
\end{proof}

For a $n$-dimensional semisimple Lie algebra $\mathfrak{g}$, its \emph{Casimir element} $\Omega$ is defined as
\begin{equation}\begin{aligned}
\Omega:=\sum_{i=1}^m X_iX_i^*,
\end{aligned}\end{equation}
where $\{X_i\}^m_{i=1}$ is a basis of $\mathfrak{g}$, and $X_i^*$ is the dual basis of $X_i$ with respect to an invariant form $B$. $\Omega$ is unique up to a factor. It is known that $\Omega$ is $\mathfrak{g}$ invariant. i.e. $[\Omega,X]=0$ for any $X\in \mathfrak{g}$ \cite[Section 6.2]{humphrey}.

It is convenient to use invariant form $\langle A,B\rangle=\frac12tr(AB)$ in our case.
Then the dual basis of $\mathfrak{sp}(2n,\mathbb{C})$ in \eqref{spn basis} are given by
\begin{equation*}\begin{aligned}
X_{i,j}^*&=X_{i,j}^-,\qquad Y_{i,j}^*=Y_{i,j}^{-},\qquad {U_i}^*=2U_i^-,\qquad H_{i,i}^*=H_{i,i},\\
{X_{i,j}^{-}}^*&=X_{i,j},\qquad{Y_{i,j}^{-}}^*=Y_{i,j},\quad\quad{U_i^-}^*=2U_i.
\end{aligned}\end{equation*}
The dual basis of $\mathfrak{sp}(2,\mathbb{C})$ are given by $U^*=2U^{-}$, ${U^{-}}^*=2U$, $H^*=H$.

By definition, the Casimir elements of $\mathfrak{sp}(2n,\mathbb{C})$ and $\mathfrak{sp}(2,\mathbb{C})$ are given by
\begin{equation}\begin{aligned}\label{casimir element}
\Omega^L=&
\sum_{i=1}^n H_{i,i}^2+\sum_{i<j}^n\bigl(X_{i,j}X_{i,j}^-+X_{i,j}^-X_{i,j}
+Y_{i,j}Y_{i,j}^{-} +Y_{i,j}^{-}Y_{i,j} \bigr)+ 2\sum_{i=1}^n (U_iU_i^{-}+U_i^{-}U_i),\\
\Omega^R=&  H^2 + 2(UU^{-} + U^{-}U),
\end{aligned}\end{equation}
respectively, which can be simplified as
\begin{equation}\begin{aligned}\label{casimir element 1}
\Omega^L=&
\sum_{i=1}^n H_{i,i}^2+\sum_{i\neq j}^n\bigl(X_{i,j}^-X_{i,j}
+ Y_{i,j}^{-}Y_{i,j} \bigr)+ \sum_{i=1}^n \left(4U_i^{-}U_i+(n+1)H_{i,i}\right),\\
\Omega^R=&  H^2  + 4U^{-}U+2H,
\end{aligned}\end{equation}
by $X_{i,j}^-=X_{j,i}$, $Y_{i,j}=Y_{j,i}$, $Y_{i,j}^-=Y_{j,i}^-$, for $i<j$ and
\begin{equation}\begin{aligned}
\left[U_i,U_i^{-}\right]&=H_{i,i},\qquad \left[U,U^{-}\right]=H,\\
\frac{1}{2}\sum_{i\neq j}[Y_{i,j},Y_{i,j}^-]&=\frac{1}{2}\sum_{i\neq j}(H_{i,i}+H_{j,j})=(n-1)\sum_{i}H_{i,i}.
\end{aligned}\end{equation}
The {\it Casimir operators} are $\triangle^L:=\pi(\Omega^L)$, $\triangle^R:=\pi(\Omega^R)$.

\begin{prop}\label{Ap 2}
The Casimir operators of ${\rm Sp}(n)$ and Sp$(1)$ are given by
\begin{equation}\begin{aligned}\label{omegaLex}
\triangle^L&=-\frac{|\mathbf{q}|^2}{2}\triangle_0+2\mathfrak{D}+(R^{0'})^2+(R^{1'})^2+2n\bigl(R^{0'}+R^{1'}\bigr),\\
\triangle^R&=4 \mathfrak{D}+(R^{0'}-R^{1'})^2+2(R^{0'}+R^{1'}),
\end{aligned}\end{equation}
respectively, where $\mathfrak{D}$ is given by \eqref{def D R}.
\end{prop}

The proof will be given in Appendix.

\begin{cor}\label{cor ra tan}
\begin{equation*}\begin{aligned}\label{ra ta form}
\triangle= (1-|\mathbf{q}|^2)\bigg[\frac{1+|\mathbf{q}|^2}{|\mathbf{q}|^2}\triangle^R-\frac{2}{|\mathbf{q}|^2}\triangle^L+\frac{1-|\mathbf{q}|^2}{|\mathbf{q}|^2}(R^{0'}+R^{1'})^2 +\frac{4n-2+2|\mathbf{q}|^2}{|\mathbf{q}|^2}(R^{0'}+R^{1'}\bigr)\bigg].\\
\end{aligned}\end{equation*}
\end{cor}
\begin{proof}

By $ \eqref{omegaLex}$, we get
\begin{equation}\begin{aligned}\label{D}
\mathfrak{D}&=\frac{1}{4}\triangle^R-\frac{1}{4}(R^{0'}-R^{1'})^2-\frac{1}{2}(R^{0'}+R^{1'}),\\
\frac12\left(\triangle^R-\triangle^L\right) &= \frac{|\mathbf{q}|^2}{4}\triangle_0+\mathfrak{D}-R^{0'}R^{1'}-(n-1)\bigl(R^{0'}+R^{1'}\bigr).
\end{aligned}\end{equation}
Their summation gives us
\begin{equation}\begin{aligned}\label{R-L}
\frac{|\mathbf{q}|^2}{4}\triangle_0&=\frac14\triangle^R-\frac12\triangle^L+\frac14 ( R^{0'}+R^{1'})^2+\left(n-\frac12\right)\bigl(R^{0'}+R^{1'}\bigr).
\end{aligned}\end{equation}
Substituting first identity in \eqref{D} and \eqref{R-L} into \eqref{ex form}, we get the result.
\end{proof}

\subsection{The action of ${\rm Sp}(n){\rm Sp}(1)$-invariant operators on $\mathcal{H}_{a,b}$}
\begin{lem}\label{lem eigen}
For any $\phi\in \mathcal{H}_{a,b}$, we have
\begin{equation}\begin{aligned}\label{en pab}
\triangle^L\phi=&\left(a^2+b^2+2na+2(n-1)b\right)\phi,\\
\triangle^R\phi=&\left((a-b)^2+2(a-b)\right)\phi,\\
\triangle \phi=&4(1-|\mathbf{q}|^2)a( 1-b)\phi.
\end{aligned}\end{equation}
\end{lem}

\begin{proof}
Since these operators are all ${\rm Sp}(n){\rm Sp}(1)$-invariant and $\mathcal{H}_{a,b}$ is irreducible, it is sufficient to check the lemma for $\phi$ to be the highest weight vector $p_{a,b}$ by Schur's Lemma.

Since $\pi(X_{i,j}),\pi(Y_{i,j}),\pi(U_i)$ annihilate $p_{a,b}$ for $i<j$ by \eqref{pi pab}-\eqref{3.24}, we use brackets
\begin{equation*}\begin{aligned}
\left[X_{i,j},X_{i,j}^{-}\right]=H_{i,i}-H_{j,j},\quad\left[Y_{i,j},Y_{i,j}^{-}\right]=H_{i,i}+H_{j,j},\quad\left[U_{i},U_{i}^{-}\right]=H_{i,i}
\end{aligned}\end{equation*}
to rewrite $\Omega^L$ in \eqref{casimir element} as
\begin{equation*}\begin{aligned}
\sum_{i} H_{i,i}^2+2\sum_{i<j}\bigl(X_{i,j}^-X_{i,j}
 +Y_{i,j}^{-}Y_{i,j} \bigr)+ 4\sum_{i} U_i^{-}U_i+2nH_{1,1}+2(n-1)H_{2,2},\;{\rm mod}\; H_{3,3},\;\dots H_{n,n}.
\end{aligned}\end{equation*}
This together with \eqref{pi pab}-\eqref{pi H11H22} implies
\begin{equation}\begin{aligned}
\triangle^Lp_{a,b}=&\left( \pi(H_{1,1})^2 + \pi(H_{2,2})^2+2n\pi(H_{1,1})+2(n-1)\pi(H_{2,2}) \right)p_{a,b}\\
=&\left(a^2+b^2+2na+2(n-1)b\right)p_{a,b}.
\end{aligned}\end{equation}
Using \eqref{piH}, we get the second formula in \eqref{en pab} by the formula for $\triangle^R$ in \eqref{casimir element}. Hence,
\begin{equation*}\begin{aligned}
\triangle p_{a,b}=&(1-|\mathbf{q}|^2)\bigg[\frac{1+|\mathbf{q}|^2}{|\mathbf{q}|^2}\left((a-b)^2+2(a-b)\right)-\frac{2}{|\mathbf{q}|^2}\left(a^2+b^2+2na+2(n-1)b\right)\\ &\qquad\qquad\quad+\frac{1-|\mathbf{q}|^2}{|\mathbf{q}|^2}(a+b)^2+\frac{4n-2+2|\mathbf{q}|^2}{|\mathbf{q}|^2}(a+b)\bigg]p_{a,b}\\
=&4(1-|\mathbf{q}|^2)a( 1-b)p_{a,b}\\
\end{aligned}\end{equation*}
by the formula of $\triangle$ in Corollary \ref{cor ra tan}.
\end{proof}

\section{The expansion of $\mathcal{QM}$-harmonic functions}
\subsection{The hypergeometric equation satisfied by the radial parts}
\begin{lem}\label{int by part}
For $X\in \mathfrak{sp}(n)\mathfrak{sp}(1)$ and $f,g\in C^2(S^{4n-1})$, we have
\begin{equation*}\begin{aligned}
\int_{S^{4n-1}}\pi(X) f(\eta)g(\eta)d\sigma(\eta)=-\int_{S^{4n-1}}f(\eta)\pi(X) g(\eta)d\sigma(\eta).
\end{aligned}\end{equation*}
In particular, $\Omega^L$ and $\triangle^R$ are self-adjoint operators on $S^{4n-1}$.
\end{lem}
\begin{proof}
Let $h_t:=e^{tX}$ with $X\in \mathfrak{sp}(n)$. Then, by definition \eqref{representation defd} we have
\begin{equation*}\begin{aligned}
\int_{S^{4n-1}}\pi(X) f(\eta)g(\eta)d\sigma(\eta)=&\lim_{t\rightarrow 0}\int_{S^{4n-1}} \frac{d}{dt}f(h_t^{-1}\cdot \eta)g(\eta)d\sigma(\eta)\\
=& \lim_{t\rightarrow 0}\frac{d}{dt}\int_{S^{4n-1}} f(h_t^{-1}\cdot\eta)g(\eta)d\sigma(\eta)\\
=&\lim_{t\rightarrow 0}\frac{d}{dt}\int_{S^{4n-1}}f(\zeta) g(h_{t}\zeta)d\sigma(\zeta)\\
=&-\int_{S^{4n-1}}f(\zeta)\pi(X) g(\zeta)d\sigma(\zeta).
\end{aligned}\end{equation*}
Here we take an orthogonal transformation $\zeta=h_t^{-1}\eta$, which preserves the measure $d\sigma$ by ${\rm Sp}(n){\rm Sp}(1)\subset SO(4n)$. It is same for $X'\in$Sp$(1)$. The lemma is proved.
\end{proof}

Recall that the hypergeometric function ${}_2F_1(\alpha,\beta,\gamma;t)$ is defined as
\begin{equation}\begin{aligned}\label{2F1}
{}_2F_1(\alpha,\beta,\gamma;t):=\sum_{n=0}^{\infty}\frac{(\alpha)_n(\beta)_n}{(\gamma)_n n!}t^n,
\end{aligned}\end{equation}
 where $(\alpha)_n=\frac{\Gamma(n+\alpha)}{\Gamma(\alpha)}$, i.e. $(\alpha)_0=1$, $(\alpha)_n=\alpha\cdot(\alpha+1)\dots(\alpha+n-1)$. It satisfies the {\it hypergeometric equation}
\begin{equation}\begin{aligned}\label{hypergeo eq}
t(1-t)u''(t)+ (\gamma-(\alpha+\beta+1)t)u'(t)- \alpha\beta u(t)=0.
\end{aligned}\end{equation}
\begin{prop}\cite[Section 2.1.3 (12), Section 2.2.2]{Er}\label{hypergeo}
Suppose that $\alpha,\beta\in \mathbb{C}$ and $\gamma$ satisfying $\gamma-\alpha$ or $\gamma-\beta \notin \mathbb{N}^{-}\cup\{0\}$. Then, the hypergeometric equation \eqref{hypergeo eq} has two independent solutions $u_1$ and $u_2$ such that: (1) $u_1(t)={}_2F_1(\alpha,\beta,\gamma;t)$ absolutely converges on $[0,1)$;

(2) If $\alpha,\beta,\gamma\in \mathbb{N}$ with Re $(\gamma-\alpha-\beta)\geq 1$, then
\begin{equation}\begin{aligned}
u_2(t)={}_2F_1(\alpha,\beta,\alpha+\beta+1-\gamma;1-t)
\end{aligned}\end{equation}
converges on $(0,1)$, but diverges at $t=0$; If $\beta,\gamma\in \mathbb{N}$, $\alpha\in \mathbb{N}^{-}$ with Re $(\gamma-\alpha-\beta)\geq 1$, then
\begin{equation}\begin{aligned}
u_2(t)=t^{1-\gamma}{}_2F_1(2+\alpha-\gamma,2+\beta-\gamma,2-\gamma;t)
\end{aligned}\end{equation}
converges on $(0,1)$, but diverges at $t=0$.
\end{prop}

\begin{proof}[Proof of Theorem \ref{no boundary decomp}]
For fixed $a,b$, let $P_{a,b}$ be the orthogonal projection from $L^2(S^{4n-1})$ onto $\mathcal{H}_{a,b}$, and let $\{\varphi_{a,b}^{(j)}\}$ be an orthonormal basis of $\mathcal{H}_{a,b}$. Then, $P_{a,b}$ is the integral operator
\begin{equation}\begin{aligned}\label{decom pab}
P_{a,b}f(\zeta):=\int_{S^{4n-1}}K_{a,b}(\zeta,\eta)f(\eta)d\sigma(\eta),
\end{aligned}\end{equation}
for $f\in L^2(S^{4n-1})$, with the kernel
\begin{equation*}\begin{aligned}
K_{a,b}(\zeta,\eta):=\sum_{j=1}^{m_{a,b}}\varphi_{a,b}^{(j)}(\zeta)\overline{\varphi_{a,b}^{(j)}(\eta)},\qquad \zeta,\eta\in S^{4n-1},
\end{aligned}\end{equation*}
where $m_{a,b}={\rm dim} \mathcal{H}_{a,b}$.
By applying the decomposition of $L^2(S^{4n-1})$ in Theorem A to $u(r\cdot)\in L^2(S^{4n-1})$ for fixed $r\in(0,1)$, we get $u(r\zeta)=\sum_{a,b}u_{a,b}(r\zeta)$, with
\begin{equation}\begin{aligned}\label{f lambda}
u_{a,b}(r\zeta)&:=\int_{S^{4n-1}}K_{a,b}(\zeta,\eta)u(r\eta)d\sigma(\eta).
\end{aligned}\end{equation}

    On the other hand, it is direct to check that
\begin{equation}\begin{aligned}
R^{0'}+R^{1'}=\sum_{A,A'}z_A^{A'}\nabla_A^{A'}u(r\zeta)=r\frac{du}{dr}(r\zeta),
\end{aligned}\end{equation}
is the Euler operator by
\begin{equation*}\begin{aligned}
\frac12(x+\mathbf{i}y)(\partial_x-\mathbf{i}\partial_y)+\frac12(x-\mathbf{i}y)(\partial_x+\mathbf{i}\partial_y)=x\partial_x+y\partial_y.
\end{aligned}\end{equation*}
Thus, the $\mathcal{QM}$-invariant Laplacian operator $\triangle$ can be rewritten as
\begin{equation}\begin{aligned}\label{radial decomp}
\triangle=&(1-r^2)\bigg[\frac{1+r^2}{r^2}\triangle^R-\frac{2}{r^2}\triangle^L+\frac{1-r^2}{r^2}\left(r\frac{d}{dr} \right)^2 +\frac{4n-2+2r^2}{r^2} r\frac{d}{dr}\bigg]\\
\end{aligned}\end{equation}
by the expression of $\triangle$ in Corollary \ref{cor ra tan}.

  Now applying radial part of $\triangle$ in $\eqref{radial decomp}$ to \eqref{f lambda} with $u$ satisfying $\triangle u=0$, we get
\begin{equation}\begin{aligned}\label{ra-tan}
&(1-r^2)\bigg(\frac{1-r^2}{r^2}\left(r\frac{d}{dr} \right)^2 +\frac{4n-2+2r^2}{r^2}\cdot r\frac{d}{dr}\bigg)u_{a,b}(r\zeta)\\
=&-(1-r^2)\int_{S^{4n-1}}K_{a,b}(\zeta,\eta)\left(\frac{1+r^2}{r^2}\triangle^R-\frac{2}{r^2}\triangle^L\right)u(r\eta)d\sigma(\eta)\\
=&-(1-r^2)\int_{S^{4n-1}}u(r\eta)\bigg(\frac{1+r^2}{r^2}\triangle^R_{\eta}-\frac{2}{r^2}\triangle^L_{\eta}\bigg)K_{a,b}(\zeta,\eta)d\sigma(\eta)\\
=&-(1-r^2)\int_{S^{4n-1}}u(r\eta)\sum_{j=1}^{m}\varphi_{a,b}^{(j)}(\zeta)\bigg( \frac{1+r^2}{r^2}\triangle^R-\frac{2}{r^2}\triangle^L\bigg)\overline{\varphi_{a,b}^{(j)}(\eta)}d\sigma(\eta)\\
=&-(1-r^2)\left[\frac{1+r^2}{r^2}\left((a-b)^2-2(a-b)\right)-\frac{2}{r^2}\left(a^2+b^2+2n(a+b)-2b\right)\right]u_{a,b}(r\zeta)
\end{aligned}\end{equation}
by using the self-adjointness of $\triangle^L$ and $\triangle^R$ on $S^{4n-1}$ in Lemma \ref{int by part} and Lemma \ref{lem eigen}. Since $\varphi_{a,b}^{(j)}(r\zeta)=r^{a+b}\varphi_{a,b}^{(j)}(\zeta)$ as a homogenous function of degree $a+b$, we can write
\begin{equation}\begin{aligned}\label{varphi ab}
u_{a,b}(r\zeta)=\sum_jr^{a+b}g_{a,b}^{(j)}(r^2)\varphi_{a,b}^{(j)}(\zeta),
\end{aligned}\end{equation}
for some function $g_{a,b}^{(j)}$ on $[0,1)$.
Moreover, $\varphi_{a,b}^{(j)}$ are linear independent. So \eqref{ra-tan} implies that
\begin{equation}\begin{aligned}\label{eq decom}
\bigg[\frac{1-r^2}{r^2}\left(r\frac{d}{dr} \right)^2 + \frac{4n-2+2r^2}{r^2} r\frac{d}{dr}&+\frac{1+r^2}{r^2}\left((a-b)^2+2(a-b)\right)\bigg.\\
&- \bigg.\frac{2}{r^2}\left(a^2+b^2+2n(a+b)-2b\right)\bigg]r^{a+b}g_{a,b}^{(j)}(r^2)=0.
\end{aligned}\end{equation}
For fixed $a,b,j$, denote $g_{a,b}^{(j)}$ by $g$ briefly. Note that radial derivative satisfy
\begin{equation}\begin{aligned}\label{radial derivatives}
r\frac{d}{dr}\left(r^{a+b}g(r^2)\right)=& r^{a+b}\left(2r^2g' + (a+b)g\right),\\
\left(r\frac{d}{dr} \right)^2\left(r^{a+b}g(r^2)\right)=& r^{a+b}\left(4r^4g''+ 4r^2(a+b+1)g' +(a+b)^2 g\right).
\end{aligned}\end{equation}
Substituting \eqref{radial derivatives} into \eqref{eq decom}, we get
\begin{equation}\begin{aligned}\label{hyper eq}
4r^2(1-r^2)g''(r^2)+ 4\bigl[(1-r^2)(a+b +1) +2n-1+r^2\bigr]g'(r^2)+\mathfrak{C}g(r^2)=0,
\end{aligned}\end{equation}
where
\begin{equation*}\begin{aligned}
\mathfrak{C}:=&\frac{1-r^2}{r^2}(a+b)^2+ \frac{4n-2+2r^2}{r^2}(a+b)+\frac{1+r^2}{r^2}\left((a-b)^2+2(a-b)\right)\\
&-\frac{2}{r^2}\left(a^2+b^2+2n(a+b)-2b\right)=- 4a(b-1 ).
\end{aligned}\end{equation*}
Thus, $g$ satisfies
\begin{equation}\begin{aligned}\label{g eq}
r^2(1-r^2)g''(r^2)+ \bigl(2n+a+b-(a+b)r^2\bigr)g'(r^2)- a(b-1)g(r^2)=0,\\
\end{aligned}\end{equation}
which is the standard hypergeometric equation \eqref{hyper eq} with $t=r^2$, $\alpha=a$, $\beta=b-1$, $\gamma=2n+a+b$ and Re $(\gamma-\alpha-\beta)=2n+1>1$.
Since $g$ is continuous on $[0,1)$, we must have
\begin{equation*}\begin{aligned}
g^{(j)}_{a,b}(r^2) ={}_2F_1\bigl(a,\; b-1,\; 2n+a+b;\; r^2\bigr)
\end{aligned}\end{equation*}
up to a factor $C$ by Proposition \ref{hypergeo}.
\end{proof}
\subsection{The convergence}
 Now we show that \eqref{hermitian} uniformly and absolutely converges by the method in \cite{Stein}. For small $\varepsilon$, it is easy to see that $\triangle$ is uniformly elliptic by its expression \eqref{ex form} on the ball $B_{\varepsilon}^{4n}$ with radius $\varepsilon$ centered at the origin. Thus, if $u$ satisfies $\triangle u=0$ on $B^{4n}_{\varepsilon}$, then $u$ must be $C^{\infty}(B^{4n}_{\varepsilon})$ by elliptic regularity (cf. e.g. \cite[lemma 6.1]{Gilbrag}). For given $\mathbf{a}\in B^{4n}$, since $\triangle(u\circ \varphi_{\mathbf{a}})=0$ by invariance, we must have $u\circ \varphi_{\mathbf{a}}\in C^{\infty}$ near the origin. This together with the smoothness of $\varphi_{\mathbf{a}}$ implies $u$ is $C^{\infty}$ near $\mathbf{a}$. Hence $u$ is $C^{\infty}$ on $B^{4n}$.

  For $0\leq r<1$, we can write decomposition \eqref{varphi ab} of $u$ as
\begin{equation}\begin{aligned}\label{urzeta}
u(r\zeta)=:\sum_{a\geq 2b}F_{a,b}(r^2)r^{a+b}u_{a,b}(\zeta),
\end{aligned}\end{equation}
with some $u_{a,b}\in \mathcal{H}_{a,b}$, where $F_{a,b}(r^2):={}_2F_1\bigl(a,\; b-1,\; 2n+a+b;\; r^2\bigr)$. But $u_{a,b}$ is a harmonic polynomial of homogeneous degree $(a+b)$ satisfying
\begin{equation}\begin{aligned}\label{tris}
\triangle_{0,S}u_{a,b}(\zeta):=C_{a,b}u_{a,b}(\zeta),\qquad{\rm where}\qquad C_{a,b}:=-(a+b)(a+b+4n-2),
\end{aligned}\end{equation}
by \cite[Section 3.1.4]{Stein}, where $\triangle_{0,S}$ is the spherical Euclidean Laplacian on $S^{4n-1}$. So
\begin{equation}\begin{aligned}\label{es uab}
\int_{S^{4n-1}}\left(\triangle_{0,S}\right)^{k} u(r_0\zeta)\overline{u_{a,b}(\zeta)}d\sigma(\zeta)=&\int_{S^{4n-1}}u(r_0\zeta)\left(\triangle_{0,S}\right)^{k}\overline{u_{a,b}(\zeta)}d\sigma(\zeta)\\
=&C_{a,b}^{k}F_{a,b}(r_0^2)r_0^{a+b}||u_{a,b}||_{L^2(S^{4n-1})}^2,
\end{aligned}\end{equation}
for $0<r_0<1$, by \eqref{urzeta}-\eqref{tris}.
Now applying the Cauchy-Schwartz inequality to the left hand side of \eqref{es uab}, we get
\begin{equation}\begin{aligned}\label{L2 es}
||u_{a,b}||_{L^2(S^{4n-1})}\leq\frac{||u||_{C^{2k}(B^{4n}_{r_0})}}{F_{a,b}(r_0^2)r_0^{a+b}(a+b+1)^{2k}}.
\end{aligned}\end{equation}
Denote by $\mathcal{B}(\zeta,\varepsilon)$ the Euclidean ball of radius $\varepsilon$ centered at $\zeta\in S^{4n-1}$. By the mean value formula for harmonic functions, we get
\begin{equation}\begin{aligned}\label{L1 es}
|u_{a,b}(\zeta)|=&\frac{1}{|(\mathcal{B}(\zeta,\varepsilon))|}\left|\int_{\mathcal{B}(\zeta,\varepsilon)}u_{a,b}dV\right|
\leq \frac{||u_{a,b}||_{L^2(\mathcal{B}(\zeta,\varepsilon))}}{|(\mathcal{B}(\zeta,\varepsilon))|^{\frac12}}\leq\frac{||u_{a,b}||_{L^2(\mathcal{B}(0,1+\varepsilon))}}{|(\mathcal{B}(\zeta,\varepsilon))|^{\frac12}}\\
\leq& \frac{1}{|(\mathcal{B}(\zeta,\varepsilon))|^{\frac12}}\left(\omega_{4n}\int_0^{1+\varepsilon}\frac{1}{4n}\rho^{2a+2b+4n-1}d\rho\int_{S^{4n-1}}|u_{a,b}(\zeta)|^2 d\sigma\right)^{\frac12}\\
\lesssim&\frac{(1+\varepsilon)^{a+b+2n}}{(a+b)\varepsilon^{2n}} ||u_{a,b}||_{L^2(S^{4n-1})}\\
\lesssim&C (a+b+1)^{2n-1}||u_{a,b}||_{L^2(S^{4n-1})},
\end{aligned}\end{equation}
if take $\varepsilon= \frac{1}{a+b}$, where $C$ is a constant independent of $a$, $b$ and $\omega_{4n}$ is the volume of unit ball in $\mathbb{R}^{4n}$.
By \eqref{L2 es}-\eqref{L1 es}, we get that for $0<r<r_0$,
\begin{equation}\begin{aligned}
\sum_{a\geq 2b}F_{a,b}(r^2)r^{a+b}|u_{a,b}(\zeta)|\lesssim\sum_{a\geq 2b}\frac{F_{a,b}(r^2)r^{a+b}}{F_{a,b}(r_0^2)r_0^{a+b}}\frac{1}{(a+b+1)^{2k-2n+1}}<\infty.
\end{aligned}\end{equation}
if we take $k>n$, since $F_{a,b}(r^2)$ is an nonnegative increasing function in $r^2$ by its definition \eqref{hypergeo}. Thus, \eqref{no boundary decomp eq} uniformly and absolutely converges on any compact subset of $B^{4n}$.

\section{The Green formula and solutions to the Dirichlet problem for $\mathcal{QM}$-invariant Laplace equation}
\subsection{The Green formula for $\mathcal{QM}$-invariant Laplacian}
Let $J_R(\psi)$ be the Jacobian determinant of a smooth transformation $\psi$ from $B^{4n}$ to itself.
\begin{lem}\label{invariant measure}
Let $\rho(\mathbf{q})=1-|\mathbf{q}|^2$. Then, (1) for $\mathbf{a}\in B^{4n},$
\begin{equation}\begin{aligned}\label{jacobian}
J_R(\varphi_{\mathbf{a}})(\mathbf{q})=\frac{\rho(\varphi_{\mathbf{a}}(\mathbf{q}))^{2n+2}}{\rho(\mathbf{q})^{2n+2}}.
\end{aligned}\end{equation}
(2) The measure $d\widetilde{V}(\mathbf{q})=\rho(\mathbf{q})^{-2n-2}dV(\mathbf{q})$ is invariant under $\mathcal{QM}$-transformations.
\end{lem}
\begin{proof}
(1) To prove \eqref{jacobian}, we first show
\begin{equation}\begin{aligned}\label{jacobian sp}
J_R(\varphi_{\mathbf{a}})(\mathbf{0})=\rho(\mathbf{a})^{2n+2}\quad{\rm \;and}\quad J_R(\varphi_{\mathbf{a}})(\mathbf{a})=\rho(\mathbf{a})^{-2n-2}.
\end{aligned}\end{equation}
Note that
\begin{equation*}\begin{aligned}
\varphi_{\mathbf{a}}(\mathbf{q})=(\mathbf{a}-\mathbf{A}\mathbf{q})(1+\mathbf{a}^*\mathbf{q}+O(|\mathbf{q}|^2))=\mathbf{a}+(\mathbf{a}\mathbf{a}^*-\mathbf{A})\mathbf{q}+O(|\mathbf{q}|^2).
\end{aligned}\end{equation*}
By \eqref{eigen}, the $n\times n$ quaternionic matrix $\mathbf{a}\mathbf{a}^*-\mathbf{A}$ has eigenvalue $|\mathbf{a}|^2-1$ and $-s$ with multiplicity $4$ and $4n-4$ respectively.
So the induced linear transformation on $\mathbb{R}^{4n}$ of $\varphi_{\mathbf{a}}$ has determinant $s^{4n+4}$. Thus, the first identity in \eqref{jacobian sp} holds. Consequently, $J_R(\varphi_{\mathbf{a}})(\mathbf{a})=\frac{1}{J_R(\varphi_{\mathbf{a}})(\mathbf{0})}$, since $\varphi_{\mathbf{a}}$ is an involution and $\varphi_{\mathbf{a}}(\mathbf{a})=\mathbf{0}$.

  Now let $\varphi_{\mathbf{a}}(\mathbf{q})=\mathbf{p}$. By Proposition 3.1, there exists $(\mathbf{U},q_0)\in {\rm Sp}(n){\rm Sp}(1)$ such that $\varphi_{\mathbf{a}}=\varphi_{\mathbf{p}}\circ (\mathbf{U},q_0)\circ\varphi_{\mathbf{q}}$. Hence, by \eqref{jacobian sp}, we get
\begin{equation*}\begin{aligned}
(J_R\varphi_{\mathbf{a}})(\mathbf{q})&=(J_R\varphi_{\mathbf{p}})(\mathbf{0})J_R( \mathbf{U},q_0)(J_R\varphi_{\mathbf{q}})(\mathbf{q})\\
&=(J_R\varphi_{\mathbf{p}})(\mathbf{0})(J_R\varphi_{\mathbf{q}})(\mathbf{q})=\bigg(\frac{1-|\mathbf{p}|^2}{1-|\mathbf{q}|^2}\bigg)^{2n+2}.\\
\end{aligned}\end{equation*}
(2) $d\widetilde{V}$ is obviously invariant under Sp$(n)$Sp$(1)$. For $f\in L^1(d\widetilde{V})$ and $\mathbf{a}\in B^{4n}$,
\begin{equation*}\begin{aligned}
\int_{B^{4n}} fd\widetilde{V}&=\int_{B^{4n}} f(\mathbf{p})\left(1-|\mathbf{p}|^2\right)^{-2n-2}dV(\mathbf{p})\\
&=\int_{B^{4n}} f(\varphi_{\mathbf{a}}(\mathbf{q}))\left(1-|\varphi_{\mathbf{a}}(\mathbf{q})|^2\right)^{-2n-2}J_R(\varphi_{\mathbf{a}})(\mathbf{q})dV(\mathbf{q})\\
&=\int_{B^{4n}} f(\varphi_{\mathbf{a}}(\mathbf{q}))d\widetilde{V}(\mathbf{q})\\
\end{aligned}\end{equation*}
by \eqref{jacobian}. The corollary is proved.
\end{proof}
It follows from the well-known divergence formula \cite[Appendix C.2]{Evans} that
\begin{equation}\begin{aligned}\label{int by part d}
\int_{B^{4n}_r}u\nabla_{A}^{A'}vdV=-\int_{B^{4n}_r}v\nabla_{A}^{A'}udV+\int_{S^{4n-1}_r}uv\nabla_{A}^{A'}\chi_r d\sigma_r,
\end{aligned}\end{equation}
where  $d\sigma_r$ is the standard surface measure on $S_r^{4n-1}$, and $\chi_r$ is the defining function of $B^{4n}_r=\{\mathbf{q};\chi_r(\mathbf{q}):=\frac{r^2-|\mathbf{q}|^2}{2r}<0\}$ satisfying $|{\rm grad}\; \chi_r|=1$ on $\partial B^{4n}_r$.
We need the following Green formula for $\mathcal{QM}$-invariant Laplacian.
\begin{prop}\label{Green}
Let $u,v\in C^2(B^{4n})$ and $r<1$. Then,
\begin{equation}\begin{aligned}\label{eq green}
\int_{B_{r}^{4n}}(u\triangle v-v\triangle u) d\widetilde{V}=2\int_{S^{4n-1}}(uR^{0'}v-vR^{1'}u)(r\zeta)\frac{r^{4n-2}d\sigma(\zeta)}{(1-r^2)^{2n}}.
\end{aligned}\end{equation}

\end{prop}
\begin{proof}
By Proposition \ref{invariant laplace} and \eqref{z jab2}, the operator $\triangle$ can be rewritten as
\begin{equation}\begin{aligned}\label{rewritten laplcaian}
\triangle=4(1-|\mathbf{q}|^2)\sum_{A,B}\left(h_{AB}\nabla_{A}^{0'}\nabla_{B}^{1'}+R^{0'}+R^{1'}\right),
\end{aligned}\end{equation}
where $h_{AB}=:J_A^B-z_{A}^{0'}z_{B}^{1'}+z_{A}^{1'}z_{B}^{0'}$. Write
\begin{equation}\begin{aligned}\label{green pro}
\frac14\int_{B^{4n}_r}u\triangle vd\widetilde{V}
=&\sum_{A,B}\int_{B^{4n}_r}uh_{AB}\nabla_{B}^{1'}\nabla_{A}^{0'}v\rho^{-2n-1}dV+\int_{B^{4n}_r}u(R^{0'}+R^{1'})v\rho^{-2n-1}dV,\\
=:&\Sigma_1(u,v)+\Sigma_2(u,v).
\end{aligned}\end{equation}
Here $\nabla_{B}^{1'}$ commutes $\nabla_{A}^{0'}$ as differential operators of constant coefficients.
By integration by parts twice, we get
\begin{equation}\begin{aligned}\label{S1+S2}
\Sigma_1(u,v)
=&-\sum_{A,B}\int_{B^{4n}_r}\bigg(h_{AB}\nabla_{B}^{1'} u\nabla_{A}^{0'}v\rho^{-2n-1}+u\nabla_{B}^{1'}h_{AB}\nabla_{A}^{0'}v \rho^{-2n-1}\\
&\qquad\qquad\quad+uh_{AB}\nabla_{A}^{0'}v\nabla_{B}^{1'} \rho^{-2n-1}\bigg)dV+\int_{S^{4n-1}_r} uh_{AB}\nabla_{A}^{0'}v\nabla_{B}^{1'} \chi_r\cdot\rho^{-2n-1}d\sigma_r\\
=&\sum_{A,B}\int_{B^{4n}_r}\bigg(v h_{AB}\nabla_{A}^{0'}\nabla_{B}^{1'}u-v \left(1-\delta_{AB}\right) z_{B}^{1'}\nabla_{B}^{1'}u-(2n+1)\rho^{-1}v\cdot h_{AB}\nabla_{B}^{1'}u\nabla_{A}^{0'} \rho\\
&\qquad\qquad\;+ u \left(1-\delta_{AB} \right)z_{A}^{0'}\nabla_{A}^{0'}v+(2n+1)\rho^{-1}u\cdot h_{AB}\nabla_{A}^{0'}v\nabla_{B}^{1'} \rho\bigg)\rho^{-2n-1}dV\\
&+\sum_{A,B}\int_{S_r^{4n-1}} \left(u\cdot h_{AB}\nabla_{A}^{0'}v\nabla_{B}^{1'}\chi_r-v\cdot h_{AB}\nabla_{B}^{1'}u\nabla_{A}^{0'} \chi_r\right)\rho^{-2n-1}d\sigma_r\\
=:&S_1+S_2.
\end{aligned}\end{equation}
Here we have used
\begin{equation}\begin{aligned}
\nabla_{A}^{0'}h_{AB}&=(-1+\delta_{AB})z_{B}^{1'},\qquad \nabla_{B}^{1'}h_{AB}&=(-1+\delta_{AB})z_{A}^{0'}.
\end{aligned}\end{equation}
Noting that
\begin{equation}\begin{aligned}\label{nabla rho}
\nabla_{A}^{0'}\rho=-\nabla_{A}^{0'}|\mathbf{q}|^2=-\sum_{C}J_A^Cz_C^{1'},\qquad \nabla_{B}^{1'}\rho=-\nabla_{B}^{1'}|\mathbf{q}|^2=-\sum_{D}z_D^{0'}J_D^B
\end{aligned}\end{equation}
by definition \eqref{z jab} and Lemma \ref{lem nabla}, we find that
\begin{equation}\begin{aligned}\label{nabla habu}
\sum_{A,B}h_{AB}\nabla_{B}^{1'}u\nabla_{A}^{0'} \rho&=-\sum_{A,B,C}(J_A^B-z_{A}^{0'}z_{B}^{1'}+z_{A}^{1'}z_{B}^{0'})J_A^Cz_C^{1'}\nabla_{B}^{1'}u\\
&=-\sum_{A,B}(1-|\mathbf{q}|^2)z_{B}^{1'}\nabla_{B}^{1'}u=-\rho R^{1'}u.
\end{aligned}\end{equation}
Here we have used
\begin{equation}\begin{aligned}\label{5.13}
J_A^C=-J_C^A,\qquad J_A^CJ_C^B=\left(J^2\right)_A^B=-\delta_A^B,
\end{aligned}\end{equation}
by \eqref{nabla rho} and definition \eqref{z jab}. Similarly,
\begin{equation}\begin{aligned}\label{nabla habv}
\sum_{A,B}h_{AB}\nabla_{A}^{0'}v\nabla_{B}^{1'} \rho=-\rho R^{0'}v.
\end{aligned}\end{equation}
Substitute \eqref{nabla habu} and \eqref{nabla habv} into $S_1$ to get
\begin{equation}\begin{aligned}\label{S1 la}
S_1=\Sigma_1(v,u)+2\int_{B^{4n}_r}(v R^{1'}u-u R^{0'}v)\rho^{-2n-1}dV.
\end{aligned}\end{equation}
Moreover, it is easy to see that $$(R^{0'}-R^{1'})\rho\equiv 0\equiv(R^{0'}-R^{1'})\chi_r$$ by \eqref{nabla rho}. By integration by parts twice for the integral in \eqref{S1 la}, we get
\begin{equation*}\begin{aligned}
S_1=&\Sigma_1(v,u)+\int_{B^{4n}_r}(v R^{1'}u-u R^{0'}v)\rho^{-2n-1}dV+\int_{S^{4n-1}_r}uv\rho^{-2n-1}(R^{1'}-R^{0'})\chi_r d\sigma_r\\
&\qquad\quad\;\;-\int_{B^{4n}_r}\left(uR^{1'}v-vR^{0'}u-(2n+1)uv\rho^{-1}(R^{1'}-R^{0'})\rho\right)\rho^{-2n-1}dV\\
=&\Sigma_1(v,u)+\int_{B^{4n}_r}\left(v\left(R^{0'}+R^{1'}\right)u-u\left(R^{0'}+R^{1'}\right)v\right)d\widetilde{V}.
\end{aligned}\end{equation*}
Similarly, by \eqref{nabla habu} and \eqref{nabla habv},
\begin{equation}\begin{aligned}
\sum_{A,B}h_{AB}\nabla_{B}^{1'}u\nabla_{A}^{0'} \chi_r=-\frac{1}{2r}\rho R^{1'}u,\qquad\sum_{A,B}h_{AB}\nabla_{A}^{0'}v\nabla_{B}^{1'} \chi_r=-\frac{1}{2r}\rho R^{0'}v,
\end{aligned}\end{equation}
and $d\sigma_r=r^{4n-1}d\sigma$ in terms of spherical coordinates, we get
\begin{equation*}\begin{aligned}
S_2=\frac12\int_{S^{4n-1}}(uR^{0'}v-vR^{1'}u )\cdot(1-r^2)^{-2n}r^{4n-2}d\sigma.
\end{aligned}\end{equation*}
Substituting $S_1$ and $S_2$ above into \eqref{S1+S2} and then to \eqref{green pro}, we get
\begin{equation}\begin{aligned}
\frac14\int_{B^{4n}_r}u\triangle vd\widetilde{V}=\Sigma_1(v,u)+\Sigma_2(v,u)+\frac12\int_{S^{4n-1}}(uR^{0'}v-vR^{1'}u )\cdot\frac{r^{4n-2}d\sigma}{(1-r^2)^{2n}}.
\end{aligned}\end{equation}
\eqref{green pro} with $u$ and $v$ exchanged gives us
\begin{equation}\begin{aligned}
\frac14\int_{B^{4n}_r}v\triangle ud\widetilde{V}=\Sigma_1(v,u)+\Sigma_2(v,u).
\end{aligned}\end{equation}
So their difference gives us \eqref{eq green}. The proposition is proved.
\end{proof}

\subsection{The maximum principle and solutions to the Dirichlet problem for $\mathcal{QM}$-invariant Laplace equation}
To show the uniqueness of solutions, we need the maximum principle.
\begin{thm}\label{max}
Suppose $\Omega$ is an open subset of $B^{4n}$ and
that $f \in C^2(\Omega)$ is $\mathcal{QM}$-subharmonic on $\Omega$ and continuous on $\overline{\Omega}$ . If $f \leq 0$ on $\partial \Omega$,
then $f \leq 0$ in $\Omega$.
\end{thm}
\begin{proof}
Set $h_{\varepsilon}(\mathbf{q}):=f(\mathbf{q})+\varepsilon|\mathbf{q}|^2$. Then, $h_{\varepsilon}(\mathbf{q})\leq \varepsilon$ on $\partial \Omega$. By the expression \eqref{rewritten laplcaian} of $\triangle$, we have
\begin{equation}\begin{aligned}\label{tri h}
\triangle h_{\varepsilon}(\mathbf{q})&=\triangle f(\mathbf{q})+4\varepsilon(1-|\mathbf{q}|^2)\left[\sum_{A,B}(J_A^B-z_{A}^{0'}z_{B}^{1'}+z_{A}^{1'}z_{B}^{0'})\nabla_A^{0'}\nabla_{B}^{1'}+R^{0'}+R^{1'}\right]|\mathbf{q}|^2\\
&=\triangle f(\mathbf{q})+4\varepsilon(1-|\mathbf{q}|^2)\left[\sum_{A,B}(J_A^B-z_{A}^{0'}z_{B}^{1'}+z_{A}^{1'}z_{B}^{0'})J_A^B+2|\mathbf{q}|^2\right]\\
&=\triangle f(\mathbf{q})+8\varepsilon n(1-|\mathbf{q}|^2)>0,
\end{aligned}\end{equation}
for all $\mathbf{q}\in\Omega$, by \eqref{z jab} and \eqref{5.13}.

Now suppose that $h_{\varepsilon}(\mathbf{a})>\varepsilon$ at some point $\mathbf{a}\in\Omega$. Since $h_{\varepsilon}\leq \varepsilon$ on $\partial\Omega$. There must exist some maximum point $\mathbf{a}\in\Omega$. So $h_{\varepsilon}\circ \varphi_{\mathbf{a}}$ has local maximum at point $\mathbf{0}$. But by \eqref{tri h}, $\triangle_0(h_{\varepsilon}\circ \varphi_{\mathbf{a}})(\mathbf{0})=\triangle h_{\varepsilon}(\mathbf{q})>0$. This is impossible for $h_{\varepsilon}\circ \varphi_{\mathbf{a}}$ has a local maximum at $0$.
Thus $h_{\varepsilon}(\mathbf{q})\leq\varepsilon$ for all $\mathbf{q}\in\Omega$. The result follows by letting $\varepsilon\rightarrow 0$.
\end{proof}
\begin{proof}[Proof of Theorem \ref{posson}]
For $g\in C^2([0,1])$, the equation $\triangle g(|\mathbf{q}|^2)=0$ is equivalent to
\begin{equation}\begin{aligned}
0=r^2(1-r^2)g''(r^2)+2ng'(r^2),
\end{aligned}\end{equation}
by \eqref{g eq} with $a=b=0$ for radial function $g$. It has a solution
\begin{equation}\begin{aligned}\label{g}
g(t)=\int_{t}^{1}\left(\frac{1-s}{s}\right)^{2n}ds.
\end{aligned}\end{equation}
For fixed $\mathbf{a}\in B^{4n}_r$, let $\Omega=B_r^{4n} \setminus \varphi_{\mathbf{a}}(B^{4n}(\mathbf{0},\epsilon))$ with $\epsilon$ sufficiently small, and let
\begin{equation}\begin{aligned}\label{Ga}
G_{\mathbf{a}}(\mathbf{q})=g(|\varphi_{\mathbf{a}}(\mathbf{q})|^2),
\end{aligned}\end{equation}
which is smooth Green function on $B^{4n} \setminus \{\mathbf{a}\}$ with singularity at $\mathbf{0}$. We also denote $G(\mathbf{q}):=g(|\mathbf{q}|^2)$.
To apply Green formula, we choose an extension $\widehat{v}\in C^2(B_r^{4n})$ such that $\widehat{v}=G_{\mathbf{a}}$ on $\overline{\Omega}$.
Thus
\begin{equation}\begin{aligned}\label{total}
0=\int_{\Omega}(u\triangle \widehat{v}-\widehat{v}\triangle u) d\widetilde{V}=\int_{B^{4n}_r}(u\triangle \widehat{v}-\widehat{v}\triangle u) d\widetilde{V}-\int_{\varphi_{\mathbf{a}}(B^{4n}(\mathbf{a},\epsilon))}(u\triangle \widehat{v}-\widehat{v}\triangle u) d\widetilde{V}.
\end{aligned}\end{equation}
Applying the Green formula in Proposition \ref{Green} to $u\circ \varphi_{\mathbf{a}}$ and $\widehat{v}\circ \varphi_{\mathbf{a}}$ on $B^{4n}(\mathbf{0},\epsilon)$, we get
\begin{equation}\begin{aligned}\label{part ii}
\int_{\varphi_{\mathbf{a}}(B^{4n}(\mathbf{0},\epsilon))}(u\triangle\widehat{v}-&\widehat{v}\triangle u) d\widetilde{V}
=\int_{B^{4n}(\mathbf{0},\epsilon)}\left(u\circ\varphi_{\mathbf{a}}\cdot(\triangle \widehat{v})\circ\varphi_{\mathbf{a}}-\widehat{v}\circ\varphi_{\mathbf{a}}\cdot(\triangle u)\circ\varphi_{\mathbf{a}}\right) d\widetilde{V}\\
=&\int_{B^{4n}(\mathbf{a},\epsilon)}\left(u\circ\varphi_{\mathbf{a}}\cdot\triangle (\widehat{v}\circ\varphi_{\mathbf{a}})-\widehat{v}\circ\varphi_{\mathbf{a}}\cdot\triangle (u\circ\varphi_{\mathbf{a}})\right) d\widetilde{V}\\
=&2\int_{S^{4n-1}}\left.\left(\widehat{v}\circ\varphi_{\mathbf{a}}\cdot R^{1'}(u\circ\varphi_{\mathbf{a}})-u\circ\varphi_{\mathbf{a}}\cdot R^{0'}(\widehat{v}\circ\varphi_{\mathbf{a}})\right)\right|_{\epsilon\zeta}\frac{\epsilon^{4n-2}d\sigma(\zeta)}{(1-\epsilon^2)^{2n}}\\
=&2\int_{S^{4n-1}}\left.\left(G\cdot R^{1'}(u\circ\varphi_{\mathbf{a}})-u\circ\varphi_{\mathbf{a}}\cdot R^{0'}G\right)\right|_{\epsilon\zeta}\frac{\epsilon^{4n-2}d\sigma(\zeta)}{(1-\epsilon^2)^{2n}},
\end{aligned}\end{equation}
by taking the transformation $\mathbf{q}\rightarrow \varphi_{\mathbf{a}}(\mathbf{q})$, and using the invariance of $\triangle$ and the invariance of the measure $d\widetilde{V}$ under $\varphi_{\mathbf{a}}$. Here in the last identity, $\widehat{v}\circ\varphi_{\mathbf{a}}=G_{\mathbf{a}}\circ\varphi_{\mathbf{a}}=G$ on $S^{4n-1}(\mathbf{0},\epsilon)$.
Now substituting \eqref{part ii} into the second term in the right hand side of \eqref{total} and applying Green formula in Proposition \ref{Green} to the first term, we get
\begin{equation*}\begin{aligned}
\int_{S^{4n-1}}\left.(uR^{0'}G-G R^{1'}u)\right|_{r\zeta}\frac{r^{4n-2}d\sigma(\zeta)}{(1-r^2)^{2n}}
=\int_{S^{4n-1}}\left.\left(u\circ\varphi_{\mathbf{a}} R^{0'}G -G R^{1'}(u\circ\varphi_{\mathbf{a}})\right)\right|_{\epsilon\zeta}\frac{\epsilon^{4n-2}d\sigma(\zeta)}{(1-\epsilon^2)^{2n}}.
\end{aligned}\end{equation*}
Note that
\begin{equation*}\begin{aligned}
\int_{S^{4n-1}}\left.G\circ R^{1'}(u\circ\varphi_{\mathbf{a}})\right|_{\epsilon\zeta}\frac{\epsilon^{4n-2}d\sigma(\zeta)}{(1-\epsilon^2)^{2n}}\rightarrow 0,
\end{aligned}\end{equation*}
as $\epsilon\rightarrow 0$, by $g(\epsilon^2)\approx \epsilon^{2-4n}$ and $R^{1'}(u\circ\varphi_{\mathbf{a}})\rightarrow 0$, since coefficients of $R^{1'}$ tends to zero. Noting that
\begin{equation*}\begin{aligned}
R^{0'}G(\epsilon\zeta)=&g'(\epsilon^2)\epsilon^2=-\frac{(1-\epsilon^2)^{2n}}{\epsilon^{4n-2}},
\end{aligned}\end{equation*}
by \eqref{z jab}, we get
\begin{equation*}\begin{aligned}
&\int_{S^{4n-1}}\left.u\circ\varphi_{\mathbf{a}}\cdot R^{0'}G\right|_{\epsilon\zeta} \frac{\epsilon^{4n-2}d\sigma(\zeta)}{(1-\epsilon^2)^{2n}}\rightarrow-u(\mathbf{a}),
\end{aligned}\end{equation*}
as $\varepsilon\rightarrow 0$. Consequently,
$$u(\mathbf{a})=-\int_{S^{4n-1}}\left.(uR^{0'}G-GR^{1'}u)\right|_{r\zeta}\frac{r^{4n-2}d\sigma(\zeta)}{(1-r^2)^{2n}}.$$

  On the other hand, $G(\mathbf{q})\approx (1-|\mathbf{q}|^2)^{2n+1}$ as $\mathbf{q}\rightarrow 1$ by definition \eqref{Ga}, we have
\begin{equation*}\begin{aligned}
\int_{S^{4n-1}}\left.GR^{1'}u\right|_{r\zeta}\frac{r^{4n-2}d\sigma(\zeta)}{(1-r^2)^{2n}}\rightarrow0,
\end{aligned}\end{equation*}
as $r\rightarrow 1$. Thus
\begin{equation}\begin{aligned}\label{ua}
u(\mathbf{a})=-\lim_{r\rightarrow 1}\int_{S^{4n-1}}\left.uR^{0'}G\right|_{r\zeta}\frac{r^{4n-2}d\sigma(\zeta)}{(1-r^2)^{2n}}.
\end{aligned}\end{equation}
Since
\begin{equation}\begin{aligned}\label{psi eq}
1-|\varphi_{\mathbf{a}}(\mathbf{q})|^2=\frac{(1-|\mathbf{a}|^2)(1-|\mathbf{q}|^2)}{|1-\mathbf{a}^*\mathbf{q}|^2},
\end{aligned}\end{equation}
we have
\begin{align*}
R^{0'} G(\mathbf{q}) &= R^{0'} \left( g\left( |\varphi_{\mathbf{a}}(\mathbf{q})|^2 \right) \right) \\
&= g'(|\varphi_{\mathbf{a}}(\mathbf{q})|^2) \cdot R^{0'}\left( \frac{(1 - |\mathbf{a}|^2)(1 - |\mathbf{q}|^2)}{|1 - \mathbf{a}^* \mathbf{q}|^2} \right),\\
&= \left( \frac{(1 - |\mathbf{a}|^2)(1 - |\mathbf{q}|^2)}{|1 - \mathbf{a}^* \mathbf{q}|^2 - (1 - |\mathbf{a}|^2)(1 - |\mathbf{q}|^2)} \right)^{2n}   \left(\frac{-|\mathbf{q}|^2 (1 - |\mathbf{a}|^2)}{|1 - \mathbf{a}^* \mathbf{q}|^2}+O(1-|\mathbf{q}|^2)\right).
\end{align*}
Substitute this identity to \eqref{ua} to get
\begin{equation}\begin{aligned}\label{poisson ua}
u(\mathbf{a})
&= -\lim_{r \to 1} \int_{S^{4n-1}}u(r\zeta) \left( \frac{(1 - |\mathbf{a}|^2)(1 - r^2)}{|1 - \mathbf{a}^* r\zeta |^2 - (1 - |\mathbf{a}|^2)(1 - r^2)} \right)^{2n} \\
&\qquad\qquad\qquad\qquad\quad \cdot  \left( \frac{-r^2(1 - |\mathbf{a}|^2)}{|1 - \mathbf{a}^* r\zeta |^2}+O(1 - r^2) \right)\frac{r^{4n-2}}{(1 - r^2)^{2n}} d\sigma(\zeta) \\
&= \int u(\zeta) \cdot \left( \frac{1 - |\mathbf{a}|^2}{|1-\mathbf{a}^*\zeta|^2} \right)^{2n+1} d\sigma(\zeta).
\end{aligned}\end{equation}

For the uniqueness, let $\widetilde{u}$ be another solution to the Dirichlet problem \eqref{posson}. Apply the maximum principle in Theorem \ref{max} to $P[\varphi]-\widetilde{u}$, which is $\mathcal{QM}$-harmonic and equal to $0$ on the boundary. So we get $P[\varphi]-\widetilde{u}\equiv0$. The theorem is proved.
\end{proof}

\section{The $\mathcal{QM}$-Poisson integrals and its nontangential convergence}
\subsection{$\mathcal{QM}$-harmonic mean value formula}
Recall that the standard Laplacian has the following well known expression \cite[theorem 4.1.3]{Rudin}:
\begin{equation*}\begin{aligned}
\triangle_0 h(0)=\lim_{r\rightarrow 0}\frac{8n}{r^2}\int_{S^{4n-1}}\left(h(r\zeta)-h(0)\right)d\sigma(\zeta),
\end{aligned}\end{equation*}
for $h\in C^2(B^{4n})$. By substituting $h=f\circ\varphi_{\mathbf{a}}$, for $f\in C^2(B^{4n})$ and $\mathbf{a}\in B^{4n}$, we get
\begin{equation*}\begin{aligned}\label{inter mean v}
\triangle f(\mathbf{a})=\lim_{r\rightarrow 0}\frac{8n}{r^2}\int_{S^{4n-1}}\left(f(\varphi_{\mathbf{a}}(r\zeta))-f(\mathbf{a})\right)d\sigma(\zeta),
\end{aligned}\end{equation*}
where $\sigma$ is defined in Theorem \ref{posson}.

For $f\in C^2(B^{4n})$, define its {\it average} over the group Sp$(n)$ as
\begin{equation}\begin{aligned}\label{eq sharp}
f^{\sharp}(\mathbf{q}):=\int_{{\rm Sp}(n)} f(A\mathbf{q})dA,
\end{aligned}\end{equation}
for $\mathbf{q}\in B^{4n}$, where $dA$ is the Harr measure on Sp$(n)$.
Then, by Proposition \ref{inter mean v}, we have
\begin{equation*}\begin{aligned}
\triangle f^{\sharp}(\mathbf{a})&=\lim_{r\rightarrow 0}\frac{4n}{r^2}\int_{S^{4n-1}}(f^{\sharp}(\varphi_{\mathbf{a}}(r\zeta))-f^{\sharp}(\mathbf{a}))d\sigma(\zeta)\\
&=\lim_{r\rightarrow 0}\frac{4n}{r^2}\int_{S^{4n-1}}d\sigma(\zeta)\int_{{\rm Sp}(n)}(f(A\varphi_{\mathbf{a}}(r\zeta))-f(A\mathbf{a}))dA\\
&=\int_{{\rm Sp}(n)}\left(\lim_{r\rightarrow 0}\frac{4n}{r^2}\int_{S^{4n-1}}(f(A\varphi_{\mathbf{a}}(r\zeta))-f(A\mathbf{a}))d\sigma(\zeta)\right)dA\\
&=\int_{{\rm Sp}(n)}\triangle (f\circ A)(\mathbf{a})dA=\int_{{\rm Sp}(n)}(\triangle f)(A\mathbf{a})dA,
\end{aligned}\end{equation*}
where in the last identity we used the Sp$(n)$-invariance of $\triangle$ in Proposition \ref{m invariant}.
Thus, $\mathcal{QM}$-harmonic function $f$ satisfies $\triangle f^{\sharp}(\mathbf{a})=0$.

\begin{prop}
$\mathcal{QM}$-harmonic function $f$ on $B^{4n}$ satisfies
\begin{equation}\begin{aligned}\label{m harmonic mv}
f(\mathbf{0})=\int_{S^{4n-1}} f(r\zeta)d\sigma(\zeta).
\end{aligned}\end{equation}
\end{prop}
\begin{proof}

Since $f^{\sharp}$ is radial, we can write $f^{\sharp}(\mathbf{q})=\varphi(|\mathbf{q}|^2)$ for some $\varphi\in C^2(0,1)$. Then by \eqref{hyper eq},
\begin{equation}\begin{aligned}\label{qe 1}
\triangle f^{\sharp}(\mathbf{q})=4(1-r^2)(r^2(1-r^2)\varphi''(r^2)+2n\varphi'(r^2))=4r^{2-4n}(1-r^2)^{2n+2}v'(r^2),
\end{aligned}\end{equation}
if we denote $v(t):=t^{2n}(1-t)^{-2n}\varphi'(t)$. This together with $\triangle f^{\sharp}(\mathbf{q})=0$ implies $v'=0$ and so $\varphi'=0$.
Hence, $\varphi$ is a constant, since $\varphi$ must be continuous at $0$ by definition. Thus, $f^{\sharp}(\mathbf{0})=\lim_{r\rightarrow 0}f^{\sharp}(r(1,0,\dots,0))=f(\mathbf{0})$.
Since Sp$(n)/$Sp$(n-1)=S^{4n-1}$,
\begin{equation}\begin{aligned}\label{int S}
f(\mathbf{0})&=f^{\sharp}(\mathbf{0})=\int_{{\rm Sp}(n)} f(r A(1,0,\dots,0))dA=\int_{{\rm Sp}(n-1)}d\widetilde{A}\int_{S^{4n-1}} f(r\widetilde{A}\zeta)d\sigma(\zeta)\\
&=\int_{{\rm Sp}(n-1)}d\widetilde{A}\int_{S^{4n-1}} f(r\zeta)d\sigma(\zeta)=\int_{S^{4n-1}} f(r\zeta)d\sigma(\zeta),
\end{aligned}\end{equation}
where $d\widetilde{A}$ is the Harr measure on Sp$(n-1)$.
\end{proof}
\subsection{The $\mathcal{QM}$-Poisson kernel}
The $\mathcal{QM}$-Poisson kernel has the following invariant formula under $\mathcal{QM}$-transformations. For $\mathbf{a},\mathbf{q}\in B^{4n}$, $\zeta\in S^{4n-1}$, we have
\begin{equation}\begin{aligned}\label{P}
P(\varphi_{\mathbf{a}}(\mathbf{q}),\zeta)=P(\mathbf{q},\varphi_{\mathbf{a}}(\zeta))P(\mathbf{a},\zeta).
\end{aligned}\end{equation}
This is because
\begin{equation*}\begin{aligned}
P(\varphi_{\mathbf{a}}(\mathbf{q}),\zeta)&=\bigg(\frac{1-|\varphi_{\mathbf{a}}(\mathbf{q})|^2}{|1-\varphi_{\mathbf{a}}(\mathbf{q})^*\zeta|^2}\bigg)^{2n+1}=\bigg(\frac{(1-|\mathbf{a}|^2)(1-|\mathbf{q}|^2)}{|1-\mathbf{a}^*\mathbf{q}|^2|1-\varphi_{\mathbf{a}}(\mathbf{q})^*\zeta|^2}\bigg)^{2n+1}\\
&=\bigg(\frac{1-|\mathbf{q}|^2}{|1-\mathbf{q}^*\varphi_{\mathbf{a}}(\zeta)|^2}\frac{1-|\mathbf{a}|^2}{|1-\mathbf{a}^*\zeta|^2}\bigg)^{2n+1}=P(\mathbf{q},\varphi_{\mathbf{a}}(\zeta))P(\mathbf{a},\zeta),
\end{aligned}\end{equation*}
by using \eqref{psi eq} and
\begin{equation*}\begin{aligned}
|1-\varphi_{\mathbf{a}}(\mathbf{q})^*\zeta|^2&=\left|1-\left((\mathbf{a}-A\mathbf{q})(1-\mathbf{a}^*\mathbf{q})^{-1}\right)^{*}\zeta\right|^2\\
&=|1-\mathbf{q}^*\mathbf{a}|^{-2}|1-\mathbf{q}^*\mathbf{a}-(\mathbf{a}^*-\mathbf{q}^*\mathbf{A}^*)\zeta|^2\\
&=|1-\mathbf{q}^*\mathbf{a}|^{-2}|1-\mathbf{a}^*\zeta-\mathbf{q}^*(\mathbf{a}-\mathbf{A}\zeta)|^2\\
&=|1-\mathbf{q}^*\mathbf{a}|^{-2}|1-\mathbf{a}^*\zeta-\mathbf{q}^*\varphi_{\mathbf{a}}(\zeta)(1-\mathbf{a}^*\zeta)|^2\\
&=|1-\mathbf{q}^*\mathbf{a}|^{-2}|1-\mathbf{a}^*\zeta|^2|1-\mathbf{q}^*\varphi_{\mathbf{a}}(\zeta)|^2,\\
\end{aligned}\end{equation*}
where $\mathbf{A}=\mathbf{A}^*$.
\begin{prop}\label{lem pf}
The $\mathcal{QM}$-Poisson kernel $P$ in \eqref{def poisson k} satisfies \\
(1)\;For fixed $\zeta\in S^{4n-1}$, $ P(\cdot,\zeta)$ is $\mathcal{QM}$-harmonic on $B^{4n}$;\\
(2)\;$P(r\eta,\zeta)=P(r\zeta,\eta)$ for all $\zeta,\eta\in S^{4n-1}$ and $0<r<1$;\\
(3)\;$\int_{S^{4n-1}}P(\cdot,\zeta)d\sigma(\zeta)\equiv1$;\\
(4)\;For fixed $\eta\in S^{4n-1}$ and $\delta>0$, $\lim_{\mathbf{q}\rightarrow\eta}\int_{|\zeta-\eta|>\delta}P(\mathbf{q},\zeta)d\sigma(\zeta)=0$.
\end{prop}
\begin{proof}
(1)\;By \eqref{P}, we get
\begin{equation*}\begin{aligned}
\triangle P(\cdot,\zeta)|_{\mathbf{a}}=\left.\triangle_0 \left(P\left(\varphi_{\mathbf{a}}(\mathbf{q}),\zeta\right)\right)\right|_{\mathbf{q}=0}=\triangle_0 (P(\mathbf{q},\varphi_{\mathbf{a}}(\zeta))|_{\mathbf{q}=0}\cdot P(\mathbf{a},\zeta).
\end{aligned}\end{equation*}
Note that there exists $U\in{\rm Sp}(n)$ such that $U^*\varphi_{\mathbf{a}}(\zeta)=\eta_0=(1,0,\dots,0)$. Then,
\begin{equation*}\begin{aligned}
\left.\triangle_0 (P(\mathbf{q},\varphi_{\mathbf{a}}(\zeta)))\right|_{\mathbf{q}=0}=&\left.\triangle_0 \left(\frac{1-|\mathbf{q}|^2}{|1-\mathbf{q}^*\varphi_{\mathbf{a}}(\zeta)|^2}\right)^{2n+1}\right|_{\mathbf{q}=0}\\
=&\left.\triangle_0 \left(\frac{1-|U\mathbf{q}|^2}{|1-(U\mathbf{q})^*\varphi_{\mathbf{a}}(\zeta)|^2}\right)^{2n+1}\right|_{\mathbf{q}=0}=\left.\triangle_0 \left(\frac{1-|\mathbf{q}|^2}{|1-\mathbf{q}^*\eta_0|^2}\right)^{2n+1}\right|_{\mathbf{q}=0}\\
=&\triangle_0 \bigg(\left(1-(2n+1)|\mathbf{q}|^2+O(|\mathbf{q}|^3)\right)(1+(4n+2){\rm Re}\;q_1-(2n+1)|q_1|^2\bigg.\\
&\left.\left.\qquad+2(2n+1)(2n+2)({\rm Re}\;q_1)^2+O(|\mathbf{q}|^3)\right)\right|_{\mathbf{q}=0}\\
=&\left.(2n+1)\triangle_0\bigg(-|\mathbf{q}|^2-|q_1|^2+(4n+4)({\rm Re}\;q_1)^2\bigg)\right|_{\mathbf{q}=0}=0,
\end{aligned}\end{equation*}
by invariance of $\triangle_0$ under $U\in{\rm Sp}(n)$. Thus $\triangle P(\cdot,\zeta)=0$.\\
(2)
It follows from definition \eqref{def poisson k}.\\
(3) It follows from (2) and $\mathcal{QM}$-harmonic mean value formula \eqref{m harmonic mv} that
\begin{equation*}\begin{aligned}
\int_{S^{4n-1}}P(r\eta,\zeta)d\sigma(\zeta)&=\int_{S^{4n-1}}P(r\zeta,\eta)d\sigma(\zeta)=P(0,\eta)\equiv1,
\end{aligned}\end{equation*}
(4) This is because
\begin{equation*}\begin{aligned}
0\leq\int_{|\zeta-\eta|>\delta}P(\mathbf{q},\zeta)d\sigma(\zeta)&=\int_{|\zeta-\eta|>\delta}\bigg(\frac{1-|\mathbf{q}|^2}{|\zeta^*-\mathbf{q}^*|^2}\bigg)^{2n+1}d\sigma(\zeta)\\
&\leq\int_{|\zeta-\eta|>\delta}\bigg(\frac{1-|\mathbf{q}|^2}{\delta^2}\bigg)^{2n+1}d\sigma(\zeta)\rightarrow0,\\
\end{aligned}\end{equation*}
as $\mathbf{q}\rightarrow\eta$. The proposition is proved.
\end{proof}

\begin{prop}\label{lim pf}
Let f be a bounded measurable function on $S^{4n-1}$. Then
\begin{equation*}\begin{aligned}
\lim_{B^{4n}\ni\mathbf{q}\rightarrow\zeta}P[f](\mathbf{q})=f(\zeta)
\end{aligned}\end{equation*}
at each $\zeta\in S^{4n-1}$ where f is continuous.

\end{prop}
\begin{proof}
Suppose $|f|\leq M$ on $S^{4n-1}$. By the continuity of $f$, given $\varepsilon>0$, there exists $\delta>0$ such that $|f(\eta)-f(\zeta)|<\varepsilon$ if $|\eta-\zeta|<\delta$.
Thus, for $\mathbf{q}$ close to $\eta$,
\begin{equation*}\begin{aligned}
|P[f](\mathbf{q})-f(\zeta)|&\leq\int_{S^{4n-1}}P(\mathbf{q},\eta)|f(\eta)-f(\zeta)|d\sigma(\eta)\\
&\leq\varepsilon+2M\int_{|\eta-\zeta|\geq\delta}P(\mathbf{q},\eta)d\sigma(\eta)\leq 2\varepsilon,
\end{aligned}\end{equation*}
by Proposition \ref{lem pf} (3)-(4).
\end{proof}

\begin{prop}
For $f\in L^1(S^{4n-1})$, we have $P[f\circ \varphi_{\mathbf{a}}]=P[f]\circ \varphi_{\mathbf{a}}$ for any $\mathbf{a}\in B^{4n}$.
\end{prop}
\begin{proof}
First we assume $f$ is continuous on $S^{4n-1}$. Then, $P[f]$ and $P[f\circ \varphi_{\mathbf{a}}]$ are $\mathcal{QM}$-harmonic by Proposition \ref{lem pf} (1), and $P[f]\circ \varphi_{\mathbf{a}}$ is also $\mathcal{QM}$-harmonic by the invariance of $\triangle$ under $\varphi_{\mathbf{a}}$ in Proposition \ref{m invariant}. Also
\begin{equation}\begin{aligned}
\lim_{\mathbf{q}\rightarrow\zeta}(P[f]\circ \varphi_{\mathbf{a}})(\mathbf{q})=f(\varphi_{\mathbf{a}}(\zeta))=\lim_{\mathbf{q}\rightarrow\zeta}P[f\circ \varphi_{\mathbf{a}}](\mathbf{q}),
\end{aligned}\end{equation}
by Proposition \ref{lim pf}. By the uniqueness of the solution to the Dirichlet problem for $\mathcal{QM}$-Laplace equation in Theorem \ref{posson}, we get $P[f\circ \varphi_{\mathbf{a}}]=P[f]\circ \varphi_{\mathbf{a}}$.
Since $C(S^{4n-1})$ is dense in $L^1(S^{4n-1})$, the identity holds for any $f\in L^1(S^{4n-1})$.
\end{proof}
\subsection{The size estimate for non-isotropic balls}
Similar to the complex case in \cite[Lemma 8.1.11]{Stoll}, we have the following distantce on $B^{4n}:d(\mathbf{a},\mathbf{b}):=|1-\mathbf{a}^*\mathbf{b}|^{\frac12}$ for $\mathbf{a},\mathbf{b}\in \overline{B}^{4n}$. In particular, it is a nonisotropic distance on the sphere $S^{4n-1}$.
\begin{lem}\label{distance}
$d$ is a {\rm Sp}$(n)${\rm Sp}$(1)$-invariant distance on $\overline{B}^{4n}$.
\end{lem}
\begin{proof}
 $d$ is Sp$(n)$Sp$(1)$-invariant, since for $A\in{\rm Sp}(n)$, $p_0\in{\rm Sp}(1)$,
\begin{equation}\begin{aligned}
d(A\mathbf{a}p_0,A\mathbf{b}p_0)=|1-(A\mathbf{a}p_0)^*A\mathbf{b}p_0|^{\frac12}=|1-\mathbf{a}^*\mathbf{b}|^{\frac12}=d(\mathbf{a},\mathbf{b}).
\end{aligned}\end{equation}

For fixed $\mathbf{b}'\in \overline{B}^{4n}$, there exists $U\in{\rm Sp}\;(n)$, such that $U\mathbf{b}'=(r,0,\dots,0)^t$ for some $0\leq r<1$. So it is sufficient to prove that for $\mathbf{a}=(a_1,\dots,a_n),\mathbf{c}=(c_1,\dots,c_n)\in B^{4n}$,
\begin{equation}\begin{aligned}\label{sim distance}
d(\mathbf{a},\mathbf{c})\leq  d(\mathbf{a},(r,0,\dots,0)^t)+d((r,0,\dots,0)^t,\mathbf{c}).
\end{aligned}\end{equation}

  Denote $\widetilde{\mathbf{a}}:=(0,a_2,\dots,a_n)^t$, $\widetilde{\mathbf{c}}:=(0,c_2,\dots,c_n)^t$, then $\mathbf{a}^*\mathbf{c}=\widetilde{\mathbf{a}}^*\widetilde{\mathbf{c}}+\overline{a}_1c_1$, and so
\begin{equation}\begin{aligned}\label{6.21}
d(\mathbf{a},\mathbf{c})^2=|1-\widetilde{\mathbf{a}}^*\widetilde{\mathbf{c}}-\overline{a}_1c_1|\leq|1-\overline{a}_1c_1|+|\widetilde{\mathbf{a}}||\widetilde{\mathbf{c}}|.
\end{aligned}\end{equation}
But
\begin{equation}\begin{aligned}\label{6.22}
|1-\overline{a}_1c_1|&=|1-r\overline{a}_1+\overline{a}_1(r-c_1)|\leq|1-r\overline{a}_1|+|1-rc_1|,\\
|\widetilde{\mathbf{a}}|^2&\leq 1-|a_1|^2\leq 1-|ra_1|^2\leq 2|1-ra_1|,\\
|\mathbf{c}|^2&\leq 2|1-rc_1|.
\end{aligned}\end{equation}
The first inequality holds because
\begin{equation*}\begin{aligned}
|r-c_1|^2-|1-rc_1|^2=-(1-r^2)(1-|c_1|^2)<0.
\end{aligned}\end{equation*}
Substituting \eqref{6.22} into \eqref{6.21}, we get
\begin{equation}\begin{aligned}
d(\mathbf{a},\mathbf{c})^2\leq\left(|1-ra_1|^{\frac12}+|1-rc_1|^{\frac12}\right)^{2},
\end{aligned}\end{equation}
i.e. \eqref{sim distance} is proved.
\end{proof}
The {\it non-isotropic ball} under the distance $d$ on the sphere is
\begin{equation}\begin{aligned}
Q(\zeta,\delta):=\{\eta\in S^{4n-1};d(\zeta,\eta)<\delta\},
\end{aligned}\end{equation}
for some $\zeta\in S^{4n-1}$ and $0<\delta<1$.
 Its volume has the following estimate.
\begin{prop}\label{delta 2n+1}
For $0<\delta<1$, $\sigma(Q(\zeta,\delta))\approx\delta^{4n+2}$.
\end{prop}
To prove this estimate, we need the following inequality:
\begin{equation}\begin{aligned}\label{1-p}
(1-|\mathbf{a}|^2)\frac{1-|\mathbf{q}|}{1+|\mathbf{q}|}\leq 1-|\mathbf{p}|^2\leq (1-|\mathbf{a}|^2)\frac{1+|\mathbf{q}|}{1-|\mathbf{q}|},
\end{aligned}\end{equation}
for $\mathbf{a},\mathbf{q}\in B^{4n}$ and $\mathbf{p}=\varphi_{\mathbf{a}}(\mathbf{q})$. It follows from applying $1-|\mathbf{q}|\leq|1-\mathbf{a}^*\mathbf{q}|\leq 1+|\mathbf{q}|$ to
\begin{equation*}\begin{aligned}
1-|\mathbf{p}|^2= (1-|\mathbf{a}|^2)\frac{1-|\mathbf{q}|^2}{|1-\mathbf{a}^*\mathbf{q}|^2},
\end{aligned}\end{equation*}
by \eqref{psi eq}. We also need the following lemma to prove the volume estimate.
\begin{lem}\label{lem proj}
Let $\mathcal{P}:\mathbf{q}=(q_1,\dots,q_n)\rightarrow q_1$ be the orthogonal projection of $\mathbb{H}^n$ onto $\mathbb{H}$. Then,
\begin{equation}\begin{aligned}\label{proj on 1}
\int_{S^{4n-1}}f\circ \mathcal{P}(\zeta)d\sigma(\zeta)=c_n\int_{B^{4}}(1-|p|^2)^{2n-3}f(p)dV_1(p),
\end{aligned}\end{equation}
for every $f\in L^1(B^4)$, where $dV_1$ is the Lebesgue measure on $B^4$, $c_n=\frac{(n-1)\omega_{4n-4}}{n\omega_{4n}}$.
\end{lem}
\begin{proof}
Since $C_c(B^4)$ is dense in $L^1(B^4)$, we can assume that $f\in C_c(B^4_{r_0})$ for some $0<r_0<1$.
Let
\begin{equation*}\begin{aligned}
I(r):=&\int_{B^{4n}_r}f\circ \mathcal{P}(\mathbf{q})dV(\mathbf{q})=4n\omega_{4n}\int_{0}^r s^{4n-1}ds\int_{S^{4n-1}}f\circ\mathcal{P}(s\zeta)d\sigma(\zeta).
\end{aligned}\end{equation*}
Then
\begin{equation}\begin{aligned}\label{I'(1)}
I'(1)=4n\omega_{4n}\int_{S^{4n-1}}f\circ\mathcal{P}(\zeta)d\sigma(\zeta).
\end{aligned}\end{equation}
On the other hand, by Fubini's theorem, for $r>r_0$,
\begin{equation*}\begin{aligned}
I(r)=&\int_{|p|<r}f (p)\bigg(\int_{|\mathbf{v}|^2<r^2-|p|^2}dV_{n-1}(\mathbf{v})\bigg)dV_1(p)\\
=&\omega_{4n-4}\int_{B^{4}}f(p)(r^2-|p|^2)^{2n-2}dV_1(p),
\end{aligned}\end{equation*}
since $f$ vanishes outside $B^{4}_{r_0}$, where $\omega_{4n-4}$ is the volume of $B^{4n-4}$.
Its derivative at $r=1$ is
\begin{equation}\begin{aligned}\label{I'(1)'}
I'(1)=(4n-4)\omega_{4n-4}\int_{B^{4}}f(p)(1-|p|^2)^{2n-3}dV_1(p).
\end{aligned}\end{equation}
The results follows from \eqref{I'(1)} and \eqref{I'(1)'}.
\end{proof}

\begin{proof}[Proof of Proposition \ref{delta 2n+1}]
Because of the invariance of the volume of $Q(\zeta,\delta)$ under ${\rm Sp}(n){\rm Sp}(1)$, we only need to check it for $\zeta=\mathbf{e}_1$.
In order to estimate the volume of $Q(\mathbf{e}_1,\delta)$, we apply  formula \eqref{proj on 1} to characteristic function $\chi_{E(\delta)}$ of
\begin{equation*}\begin{aligned}
E(\delta)=\{q_1\in \mathbb{H};|q_1|\leq1,|1-q_1|^{\frac12}<\delta\},
\end{aligned}\end{equation*}
Noting that $\chi_{E(\delta)}\circ \mathcal{P}(\zeta)=\chi_{Q(\mathbf{e}_1,\delta)}(\zeta)$ for $\zeta\in S^{4n-1}$, where $Q(\mathbf{e}_1,\delta):=\{\eta\in S^{4n-1};|1-\eta_1|^{\frac12}<\delta\}$,
we get
\begin{equation*}\begin{aligned}
\sigma(Q(\mathbf{e}_1,\delta))&=\int_{S^{4n-1}}\chi_{Q(\mathbf{e}_1,\delta)}d\sigma=\int_{S^{4n-1}}\chi_{E(\delta)}\circ \mathcal{P}d\sigma\\
&=c_n\int_{E(\delta)}(1-|q_1|^2)^{2n-3}dV(q_1),
\end{aligned}\end{equation*}
by Lemma \ref{proj on 1}.
Take the coordinates transformation $p_1:=\delta^2(1-q_1)^{-1}$, which maps $E(\delta)$ to
\begin{equation*}\begin{aligned}
\widetilde{E}(\delta):=\{p_1\in \mathbb{H};2{\rm Re}\;p_1\geq\delta^2,|p_1|\geq1\},
\end{aligned}\end{equation*}
and has Jacobian $\frac{\delta^8}{|p_1|^8}$ by
\begin{equation*}\begin{aligned}
1-|q_1|^2=1-\frac{|p_1-\delta^2|^2}{|p_1|^2}=\frac{\delta^2(2{\rm Re}\;p_1-\delta^2)}{|p_1|^2}.
\end{aligned}\end{equation*}
We get
\begin{equation*}\begin{aligned}
\frac{\sigma(Q(\mathbf{e}_1,\delta))}{\delta^{4n+2}}=c_n\int_{\widetilde{E}(\delta)}(2{\rm Re}\;p_1-\delta^2)^{2n-3}\frac{dV(p_1)}{|p_1|^{4n+2}},
\end{aligned}\end{equation*}
where the integral in the right hand side belongs to $[c_1,c_2]$ for two absolute positive constants $c_1,c_2$ independent of $\delta$.
\end{proof}
\subsection{The estimate for non-tangential maximal functions of $\mathcal{QM}$-Poisson integrals}
For $f\in C(B^{4n})$ and $\alpha>1$, the \emph{non-tangential maximal function} of $f$, denoted by $M_{\alpha}f$, is
\begin{equation*}\begin{aligned}
M_{\alpha}f(\zeta)=\sup\{|f(\mathbf{q})|;\mathbf{q}\in A_{\alpha}(\zeta)\},
\end{aligned}\end{equation*}
for $\zeta\in S^{4n-1}$, where $A_{\alpha}(\zeta)$ is the non-tangential approach region defined in \eqref{non tan}. For finite Borel measure $\mu$ on $S^{4n-1}$, let
\begin{equation*}\begin{aligned}
P[\mu](\mathbf{q}):=\int_{S^{4n-1}}P(\mathbf{q},\zeta)d\mu(\zeta).
\end{aligned}\end{equation*}
The maximal function of a Borel measure $\mu$ on $S^{4n-1}$ is
defined by
\begin{equation}\begin{aligned}
M\mu(\zeta)=\sup_{\delta>0}\frac{1}{\sigma(Q(\zeta, \delta))}|\mu|(Q(\zeta, \delta)).
\end{aligned}\end{equation}
Therefore, the maximal function of $f\in L^1(S^{4n-1})$ is $Mf(\zeta)=M(fd\sigma)(\zeta)$.
\begin{prop}\label{MAPleqMu}
For $\alpha>1$, there exists a constant $C$ such that
\begin{equation*}\begin{aligned}
M_{\alpha}P[\mu]\leq CM|\mu|,
\end{aligned}\end{equation*}
for every signed Borel measure $\mu$ on $S^{4n-1}$.
\end{prop}
\begin{proof}
Fix $\zeta\in S^{4n-1}$ such that $M\mu(\zeta)<\infty$. For fixed $\mathbf{q}\in A_{\alpha}(\zeta)$,
set
\begin{equation}\begin{aligned}\label{def beta}
\beta:=\alpha(1-|\mathbf{q}|).
\end{aligned}\end{equation}
Let $N$ be the smallest integer such that $2^N\beta>2$.
Set $Q_0:=\{\tau\in S^{4n-1};d(\tau,\zeta)<\beta^{\frac12}\}$ and
\begin{equation*}\begin{aligned}
Q_k:=\{\tau\in S^{4n-1};2^{\frac{k-1}{2}}\beta^{\frac12}\leq d(\tau,\zeta)<2^\frac{k}{2}\beta^{\frac12}\},\quad \quad k=1,\dots,N.
\end{aligned}\end{equation*}
Then, $S^{4n-1}$ is the disjoint union of $Q_k$. For $\tau\in Q_0$, we have
\begin{equation*}\begin{aligned}
P(\mathbf{q},\tau)=\frac{(1-|\mathbf{q}|^2)^{2n+1}}{|\tau^*-\mathbf{q}^*|^{4n+2}}\leq\frac{2^{2n+1}}{(1-|\mathbf{q}|)^{2n+1}},
\end{aligned}\end{equation*}
and so
\begin{equation}\begin{aligned}\label{estimate 1}
\int_{Q_0}P(\mathbf{q},\tau)d\mu(\tau)\leq\frac{(2\alpha)^{2n+1}}{\beta^{2n+1}}\int_{Q(\zeta,\beta^{\frac12})}d|\mu|\lesssim (2\alpha)^{2n+1}M|\mu|(\zeta),
\end{aligned}\end{equation}
by the estimate in Proposition \ref{delta 2n+1}.
Since for $\mathbf{q}\in A_{\alpha}(\zeta)$, we have that for $\tau\in Q_k$
\begin{equation}\begin{aligned}\label{eta-zeta}
d(\tau,\zeta)^2=|\tau-\zeta|\leq|\mathbf{q}-\zeta|+|\mathbf{q}-\tau|<\alpha(1-|\mathbf{q}|)+|\mathbf{q}-\tau|<(\alpha+1)|\mathbf{q}-\tau|,
\end{aligned}\end{equation}
i.e.
\begin{equation}\begin{aligned}\label{q-tau}
|\mathbf{q}-\tau|>(1+\alpha)^{-1}d(\tau,\zeta)^2\geq C_{\alpha}2^{k-1}\beta.
\end{aligned}\end{equation}
Substituting \eqref{def beta} and \eqref{q-tau} into the definition \eqref{def poisson k} of $\mathcal{QM}$-Poisson kernel, we obtain
\begin{equation}\begin{aligned}\label{estimate 2}
\left|\int_{Q_k}P(\mathbf{q},\tau)d\mu(\tau)\right|&\leq\int_{Q_k}\frac{(2\beta)^{2n+1}}{\alpha^{2n+1}(C_{\alpha}2^{k-1}\beta)^{4n+2}}d|\mu|\\
&\lesssim \frac{1}{(2^{k})^{2n+1}}\frac{\int_{Q(\zeta,(2^k\beta)^{\frac12})}d|\mu|}{(2^k\beta)^{2n+1}}\leq \frac{1}{2^{k(2n+1)}}M|\mu|(\zeta).
\end{aligned}\end{equation}
So the summation of \eqref{estimate 1} and \eqref{estimate 2} over $k$ is bounded by $M|\mu|(\zeta)$. The estimate is proved.
\end{proof}

\subsection{Proof of Theorem \ref{Fatou}}
For $Q:=Q(\zeta, \delta)$, denote $3Q:=\{\eta\in S^{4n-1};d(\eta,\zeta)<3\delta\}$.
We can prove the following Vitali covering lemma exactly as in \cite[lemma 5.2.3]{Rudin}.
\begin{lem}\label{vitali}
If $E$ is the union of a finite collection $\Phi$ of balls $Q \subset S^{4n-1}$, then
$\Phi$ has a disjoint subcollection $\Gamma$ such that $E \subset \bigcup_{\Gamma} 3Q$ and $\sigma(E) \leq A_3 \sum_{\Gamma} \sigma(Q)$ with $A_3=\sup_{Q}\frac{\sigma(3Q)}{\sigma(Q)}<\infty$.
\end{lem}
Using this Vitali covering lemma, we can prove the following weak-$(1,1)$ estimate for maximal function exactly as in \cite[theorem 5.2.4, theorem 5.2.7]{Rudin}.
\begin{prop}\label{lem A3}
For $f\in L^1(S^{4n-1})$. Then, $\sigma\{Mf>t\}\leq \frac{A_3}{t}||f||_{L^1(S^{4n-1})}$ for every $t>0$.
\end{prop}

For $f\in L^1(S^{4n-1})$, a point $\zeta\in S^{4n-1}$ is called a \emph{Lebesgue point} of $f$ if
\begin{equation}\label{eq:leb1}
\lim_{\delta \to 0} \frac{1}{\sigma(Q(\zeta,\delta))} \int_{Q(\zeta,\delta)} |f - f(\zeta)| \, d\sigma = 0.
\end{equation}

As in \cite{Rudin}, we can prove following property by Proposition \ref{lem A3} in the standard way by using the weak-$(1,1)$ estimate.
\begin{prop}\label{dense}
(1)\;If $f \in L^1(S^{4n-1})$, then
\begin{equation}
\label{eq:leb2}
f(\zeta) = \lim_{\delta \to 0} \frac{1}{\sigma(Q(\zeta,\delta))} \int_{Q(\zeta,\delta)} f \, d\sigma ,
\end{equation}
for any Lebesgue point $\zeta \in S^{4n-1}$ of $f$.\\
(2)\;For $f\in L^1(S^{4n-1})$, the compliment of the set of the Lebesgue points has measure zero.
\end{prop}

\begin{lem}\label{lem D=0}
Let $\mu$ be a Borel measure with $|\mu|(S^{4n-1})<\infty$. Suppose that $D|\mu|(\zeta):=\lim_{\delta\rightarrow 0}\frac{|\mu|(Q(\zeta,\delta))}{\sigma(Q(\zeta,\delta))}=0$ at a point $\zeta\in S^{4n-1}$. Then
\begin{equation*}\begin{aligned}
\lim_{\substack{\mathbf{q}\rightarrow\zeta\\ \mathbf{q}\in A_{\alpha}(\zeta)}} P[\mu](\mathbf{q})=0,
\end{aligned}\end{equation*}
for every $\alpha>0.$
\end{lem}
\begin{proof}
Since $D|\mu|(\zeta)=0$, for given $\varepsilon>0$, there exists a $\delta_0>0$ such that
\begin{equation}\begin{aligned}\label{mu es1}
|\mu|(Q(\zeta,\delta))<\varepsilon\sigma(Q(\zeta,\delta))<\varepsilon,
\end{aligned}\end{equation}
for $0<\delta<\delta_0$. Denote $\mu_0=\mu|_{Q_0}$, where $Q_0=Q(\zeta,\delta_0)$.
For $\tau\in S^{4n-1}\setminus Q_0$ and $\mathbf{q}\in A_{\alpha}(\zeta)$, we have
\begin{equation*}\begin{aligned}
\delta_0^2<|\tau-\zeta|<(\alpha+1)|\mathbf{q}-\tau|,
\end{aligned}\end{equation*}
by \eqref{eta-zeta}.
Thus,
\begin{equation*}\begin{aligned}
P(\mathbf{q},\tau)\leq\left(\frac{\alpha+1}{\delta_0^2}\right)^{4n+2}(1-|\mathbf{q}|^2)^{2n+1}.
\end{aligned}\end{equation*}
On the other hand, $M|\mu_0|(\zeta)< \varepsilon$ by \eqref{mu es1}. Thus,
\begin{equation*}\begin{aligned}
\left|P[\mu](\mathbf{q})\right|&\leq\int_{S^{4n-1}}P(\mathbf{q},\tau)d\left|\mu_0\right|+\int_{S^{4n-1}\backslash Q_0}P(\mathbf{q},\tau)d\left|\mu\right|\\
&\leq CM|\mu_0|(\zeta)+\int_{S^{4n-1}\setminus Q_0}\left(\frac{\alpha+1}{\delta_0^2}\right)^{4n+2}(1-|\mathbf{q}|^2)^{2n+1}d|\mu|\\
&\leq C\varepsilon+\int_{S^{4n-1}\setminus Q_0}\left(\frac{\alpha+1}{\delta_0^2}\right)^{4n+2}(1-|\mathbf{q}|^2)^{2n+1}d|\mu|\lesssim \varepsilon,
\end{aligned}\end{equation*}
as $A_{\alpha}(\zeta)\ni\mathbf{q}\rightarrow\zeta$. The lemma is proved.
\end{proof}

\begin{proof}[Proof of Theorem \ref{Fatou}]
 For fixed Lebesgue point $\zeta$ of $f\in L^1(S^{4n-1})$, let $\mu_{\zeta}$ be the measure on $S^{4n-1}$ defined by $\mu_{\zeta}(E):=\int_{E}|f(\eta)-f(\zeta)|d\sigma(\eta)$ for a measurable set $E$. Then
\begin{equation*}\begin{aligned}
|P[f](\mathbf{q})-f(\zeta)|=\left|\int_{S^{4n-1}}P(\mathbf{q},\eta)(f(\eta)-f(\zeta))d\sigma(\eta)\right|\leq P[\mu_{\zeta}](\mathbf{q}).
\end{aligned}\end{equation*}
But for a Lebesgue point $\zeta$ of $f$, we have
\begin{equation*}\begin{aligned}
D\mu_{\zeta}(\zeta)=\lim_{\delta\rightarrow 0^+}\frac{\mu(Q(\zeta,\delta))}{\sigma(Q(\zeta,\delta))}=\lim_{\delta\rightarrow 0^+}\frac{\int_{Q(\zeta,\delta)}|f(\eta)-f(\zeta)|d\sigma(\eta)}{\sigma(Q(\zeta,\delta))}=0.
\end{aligned}\end{equation*}
Then, we can apply Lemma \ref{lem D=0} to $\mu_{\zeta}$ to get
\begin{equation*}\begin{aligned}
\lim_{\substack{\mathbf{q}\rightarrow\zeta\\ \mathbf{q}\in A_{\alpha}(\zeta)}} |P[f](\mathbf{q})-f(\zeta)|\leq \lim_{\substack{\mathbf{q}\rightarrow\zeta\\ \mathbf{q}\in A_{\alpha}(\zeta)}} P[\mu_{\zeta}](\mathbf{q})=0.
\end{aligned}\end{equation*}
Namely, the nontangential limit of $P[f]$ at point $\zeta$ exists. Since the compliment of the set of Lebesgue points of $f$ has measure zero by Proposition \ref{dense}, nontangential limit of $P[f]$ exists almost everywhere.
\end{proof}

\appendix
\section{The Casimir operators of Sp$(n)$ and Sp$(1)$}
\begin{proof}[Proof of Proposition \ref{Ap 2}]
It is direct to check that
\begin{equation}\begin{aligned}\label{piH2}
\pi(H)^2= ({R^{0'}}-{R^{1'}})^2,\qquad \pi(U^{-}) \pi(U)= \mathfrak{D}+R^{1'},\\
\end{aligned}\end{equation}
by \eqref{u H}.
Substituting \eqref{piH2} into the expression of $\Omega^R$ in \eqref{casimir element 1}, we get the second identity in \eqref{omegaLex}.

By \eqref{Xij Yij Ui}, it is direct to check that
\begin{align*}
2\sum_{i} \pi(H_{i,i})&=2\sum_{i,A'}\left( z_{i}^{A'}\nabla_i^{A'}-z_{n+i}^{A'}\nabla_{n+i}^{A'}\right),\\
\sum_{i} \pi(H_{i,i})^2
&= \sum_{i,A',B'} \left( z_{i}^{A'} z_{i}^{B'} \nabla_{i}^{A'} \nabla_{i}^{B'}+ z_{n+i}^{A'} z_{n+i}^{B'} \nabla_{n+i}^{A'} \nabla_{n+i}^{B'} - 2 z_{i}^{A'} z_{n+i}^{B'} \nabla_{i}^{A'} \nabla_{n+i}^{B'} \right)
 +R^{0'}+R^{1'}\\
4\sum_{i} \pi(U_i^{-})\pi(U_i)&= 4\sum_{i}\left(z_{i}^{A'} z_{n+i}^{B'} \nabla_{n+i}^{A'} \nabla_{i}^{B'} + z_{n+i}^{A'} \nabla_{n+i}^{A'}\right),
\end{align*}
by using Lemma \ref{lem nabla}.
Their sum gives us
\begin{equation}\begin{aligned}\label{piHU}
\sum_{i}\left( \pi(H_{i,i})^2 \right.&\left.+ 2\pi(H_{i,i}) + 4\pi(U_i^{-})\pi(U_i)\right)
= \sum_{A,A',B'}z_{A}^{A'} z_{A}^{B'} \nabla_{A}^{A'} \nabla_{A}^{B'}\\
&+2\sum_{i,A',B'}\left(z_{i}^{A'} z_{n+i}^{B'} \nabla_{n+i}^{A'} \nabla_{i}^{B'}+z_{i}^{[A'} z_{n+i}^{B']} \nabla_{n+i}^{A'} \nabla_{i}^{B'}\right)
+ 3 \bigl(R^{0'}+R^{1'}\bigr),
\end{aligned}\end{equation}
where we use
\begin{equation}\begin{aligned}\label{znabla}
\sum_{A',B'}z_{j}^{A'} z_{n+j}^{B'} \nabla_{n+i}^{A'} \nabla_{i}^{B'}-\sum_{A',B'}z_{j}^{A'} z_{n+j}^{B'} \nabla_{i}^{A'} \nabla_{n+i}^{B'}=\sum_{A',B'}z_{j}^{[A'} z_{n+j}^{B']} \nabla_{n+i}^{A'} \nabla_{i}^{B'},
\end{aligned}\end{equation}
by relabelling indices $A',B'$ for the second summation. On other hand, it follows from definition \eqref{Xij Yij Ui}  that
\begin{equation}\begin{split}\label{piXXYY}
\sum_{i\neq j}\pi(X_{i,j}^-)\pi(X_{i,j})
=\sum_{A',B',i\neq j}\Bigl( z_{i}^{A'} z_{j}^{B'} \nabla_{j}^{A'} \nabla_{i}^{B'}
+ z_{n+j}^{A'} &z_{n+i}^{B'}\nabla_{n+i}^{A'} \nabla_{n+j}^{B'}- 2z_{i}^{A'} z_{n+i}^{B'} \nabla_{j}^{A'} \nabla_{n+j}^{B'}\Bigr) \\
&+(n-1)\sum_{j,A'}\Bigr(z_j^{A'}\nabla_{j}^{A'}+z_{n+j}^{A'}\nabla_{n+j}^{A'}\Bigr),\\
\end{split}\end{equation}
and
\begin{equation}\begin{aligned}\label{pi2}
\sum_{i\neq j} \pi(Y_{i,j}^{-})\pi(Y_{i,j})
=&\sum_{A',B',i\neq j} \Bigl( 2z_{j}^{A'} z_{n+i}^{B'} \nabla_{n+i}^{A'} \nabla_{j}^{B'}
+ 2z_{j}^{A'} z_{n+j}^{B'} \nabla_{n+i}^{A'} \nabla_{i}^{B'}\Bigr)+2(n-1)\sum_{i,A'}z_{n+i}^{A'}\nabla_{n+i}^{A'},
\end{aligned}\end{equation}
by suitably relabelling indices.
Noting that
\begin{equation}\begin{aligned}\label{AneqB}
(n-1)\sum_{i}\pi(H_{i,i})=&(n-1)\sum_{i}(z_i^{A'}\nabla_{i}^{A'}-z_{n+i}^{A'}\nabla_{n+i}^{A'}),\\
\sum_{A\neq B,A',B'}z_{A}^{A'} z_{B}^{B'} \nabla_{B}^{A'} \nabla_{A}^{B'}=&\sum_{i\neq j,A',B'}\Bigl( z_{i}^{A'} z_{j}^{B'} \nabla_{j}^{A'} \nabla_{i}^{B'}
+ z_{n+j}^{A'} z_{n+i}^{B'} \nabla_{n+i}^{A'} \nabla_{n+j}^{B'}+ 2z_{j}^{A'} z_{n+i}^{B'} \nabla_{n+i}^{A'} \nabla_{j}^{B'}\Bigl)\\
&\qquad\qquad+2\sum_{i}z_{i}^{A'} z_{n+i}^{B'} \nabla_{n+i}^{A'} \nabla_{i}^{B'}.
\end{aligned}\end{equation}
Applying \eqref{AneqB} to the summation of \eqref{piXXYY} and \eqref{pi2}, together with \eqref{znabla}, we get
\begin{equation}\begin{aligned}\label{piXY}
&\sum_{i\neq j}\left[(\pi(X_{i,j}^-)\pi(X_{i,j})
+ \pi(Y_{i,j}^{-})\pi(Y_{i,j})\right]+2\sum_{i,A',B'}\left[z_{i}^{A'} z_{n+i}^{B'} \nabla_{n+i}^{A'} \nabla_{i}^{B'}+(n-1)\pi(H_{i,i})\right] \\
=&\sum_{A\neq B,A',B'}z_{A}^{A'} z_{B}^{B'} \nabla_{B}^{A'} \nabla_{A}^{B'}+2\sum_{i\neq j,A',B'}z_{j}^{[A'} z_{n+j}^{B']} \nabla_{n+i}^{A'}\nabla_{i}^{B'}
+ 2(n-1)\bigl(R^{0'}+R^{1'}\bigr).
\end{aligned}\end{equation}
Now the summation of \eqref{piHU} and \eqref{piXY} gives us
\begin{equation}\begin{aligned}
\triangle^L=&\sum_{A, B,A',B'}z_{A}^{A'} z_{B}^{B'} \nabla_{B}^{A'} \nabla_{A}^{B'}+2\sum_{i, j, A',B'}z_{j}^{[A'} z_{n+j}^{B']} \nabla_{n+i}^{A'}\nabla_{i}^{B'}
+(2n+1)\bigl(R^{0'}+R^{1'}\bigr)\\
=&2\mathfrak{D}-\frac{|\mathbf{q}|^2}{2}\triangle_0+(R^{0'})^2+(R^{1'})^2+ 2n\bigl(R^{0'}+R^{1'}\bigr),
\end{aligned}\end{equation}
by definition of $\triangle_L$ in \eqref{casimir element 1} and $\mathfrak{D}$ in \eqref{def D R},
\begin{equation}\begin{aligned}
\sum_{i,j,A',B'}z_{j}^{[A'} z_{n+j}^{B']} \nabla_{n+i}^{A'}\nabla_{i}^{B'}
=&\sum_{i,j,A',B'}(z_{j}^{A'} z_{n+j}^{B'}-z_{j}^{B'} z_{n+j}^{A'}) \nabla_{n+i}^{A'}\nabla_{i}^{B'}\\
=&\sum_{i,A,B,A',B'}z_{A}^{A'} z_{B}^{B'}J_A^B\nabla_{n+i}^{A'}\nabla_{i}^{B'}\\
=&\sum_{i,A',B'}\varepsilon^{A'B'}|\mathbf{q}|^2\nabla_{n+i}^{A'}\nabla_{i}^{B'}=-\frac{|\mathbf{q}|^2}{4}\triangle_0,
\end{aligned}\end{equation}
by definition \eqref{def J} of $J$ and \eqref{z jab}-\eqref{z jab2}. \eqref{omegaLex} is proved.
\end{proof}

\section*{Declarations}
\subsection*{Funding}
This work is partially supported by the National Natural Science Foundation of China under Grant No. 12471080.


\begin{thebibliography}{99}
\bibitem{Ahern}
Ahern, P., Bruna, J. and Cascante, C.,
\newblock{$H^p$-theory for generalized $M$-harmonic functions in the unit ball}, {\it Indiana Univ. Math. J.} \textbf{45} (1996), 103-145.

\bibitem{Ahmed}
Ahmed, A. and Khalfallah, A.,
\newblock{Riesz-Fej\'{e}r inequalities for hyperbolic harmonic functions in the unit ball}, {\it J. Math. Anal. Appl.} \textbf{563} (2026), Ariticle 130826.


\bibitem{Co}
Ahrens, J., Cowling, M. G., Martini, A., M\"{u}ller, D.,
\newblock{Quaternionic spherical harmonics and a sharp multiplier theorem on quaternionic spheres}, {\it Math. Z.} \textbf{294} (2020), 1659-1686.


\bibitem{Burgeth}
Burgeth, B.
\newblock{A Schwarz Lemma for harmonic and hyperbolic-harmonic functions in higher dimensions}, {\it Manuscr. Math.} \textbf{77} (1992), 283-291.

\bibitem{Chang3}
Chang, D.-C., Markina, I. and  Wang, W., \newblock {On the Cauchy-Szeg\"o kernel for quaternion Siegel upper half-space}, \emph{Complex Anal. Oper. Theory} \textbf{7} (2013), 1623-1654.

\bibitem{CHANG}
Chang, D.-C., Duong, X. T., Li, J., Wang, W. and Wu, Q. Y., \newblock{An explicit formula of Cauchy-Szeg\"{o} kernel for quaternionic Siegel upper half space and applications}, {\it  Indiana Univ. Math. J.} \textbf{70} (2021), 2451-2477.

\bibitem{Chen J}
Chen, J., Chen, S., Huang, M. and Zheng, H., \newblock{Isoperimetric type inequalities for mappings induced by weighted Laplace differential operators}, {\it  J. Geom. Anal.} \textbf{33} (2023).


\bibitem{Er}
Erd{\'e}lyi, A., Magnus, W., Oberhettinger, F. and Tricomi, F. G.,
{\it Higher transcendental functions. {Vol}. {I}}, {\rm Bateman Manuscript Project} (1953).


\bibitem{Essen}
Ess\'{e}n, M., Wulan, H., S. and Xiao J.,
\newblock{Several function-theoretic characterizations of M\"{o}bius invariant \(\mathcal Q_K\)
 spaces},
{\it J. Funct. Anal. } \textbf{230} (2006), 78-115.

\bibitem{Evans}
Evans, L., C.
{\it Partial differential equations. 2nd ed.}, {\rm Providence, RI: AMS} (2010).

\bibitem{hr}
Huang, T., Wang, R., W.,
\newblock{The Cauchy-Szeg\"{o} kernel for the Hardy space of 0-regular functions on the quaternionic Siegel upper half space},
{\it  Anal. Math. Phys.} \textbf{12} (2022), article number 141.
\bibitem{humphrey}
Humphreys., J. E.,
{\it Introduction to Lie algebras and representation theory}, {\rm Springer-Verlag, Berlin, New York} (1972).


\bibitem{Liu1}
Liu, C. W. and Peng., L. H.,
\newblock{Boundary regularity in the Dirichlet problem for the invariant Laplacians $\triangle_{\gamma}$
 on the unit real ball},
{\it Proc. Am. Math. Soc. } \textbf{132} (2004), 3259-3268.

\bibitem{Liu2}
Liu, C. W. and Shi., J. H.,
\newblock{Invariant mean-value property and $\mathcal{M}$-harmonicity
in the unit ball $\mathbb{R}^n$},
{\it Acta. Math. Sinica.} \textbf{19} (2003), 187-200.

\bibitem{Liu3}
Liu, C. W. and Xu., H.,
\newblock{Lipschitz continuity of the solutions to the Dirichlet problems for the invariant Laplacians},
{\it J. Math. Anal. Appl. } \textbf{538} (2024), Article: 128447.

\bibitem{Folland}
Folland, G. B.,
\newblock{Spherical harmonic expansion of the Poisson-Szeg\"{o} kernel for the ball}, {\it  Pro. Amer. Math. Soc.} \textbf{47} (1975), 401-408.

\bibitem{Fulton}
Fulton, W. and Harris, J.
{\it Representation Theory: a first course}, {\rm Springer, New York} (2004).
\bibitem{Gilbrag}
Gilbrag, D. and Trudinger, N., S.
{\it Elliptic partial differential equations of second order}, , {\rm Springer, New York} (1998).

\bibitem{Lu}
Flynn, J., Lu, G. and Yang Q. H.,
\newblock{Sharp Hardy-Sobolev-Maz'ya, Adams and Hardy-Adams inequalities on quaternionic hyperbolic spaces and the Cayley hyperbolic plane}, {\it Adv. Math.} \textbf{319} (2017), 567-598.

\bibitem{Olo}
Olofsson, A.,
\newblock{Lipschitz continuity for weighted harmonic functions in the unit disc},  {\it Complex Var. Elliptic Equ.} \textbf{65} (2020), 1630-1660.

\bibitem{Ou2}
Ouyang, C., Yang,  W., and Zhao, R.,
\newblock{M\"{o}bius invariant $Q_p$ spaces associated with the Green's
function on the unit ball of $\mathbb{C}^n$},  {\it  Pacific J. Math.} \textbf{182} (1998), 69-99.

\bibitem{Rudin}
Rudin, W., {\it Function theory in the unit ball of $\mathbb{C}^n$}, {\rm Springer-Verlag, Berlin, New York} (1980).

\bibitem{Stein}
Stein, E., M.,
{\it Singular integrals and differentiablity properties of functions}, {\rm Princeton university press}, Princeton (1970).

\bibitem{Stoll}
Stoll, M.,
{\it Harmonic and subharmonic function theory on the hyperbolic ball}, {\rm  Cambridge University Press, Cambridge} (2016).

\bibitem{wang3}
Wan, D. and Wang, W.,
\newblock{On quaternionic Monge$-$Amp\`{e}re operator, closed positive currents and Lelong-Jensen type formula on the quaternionic space},
{\it Bull. Sci. Math.} \textbf{141} (2017), 267-311.

\bibitem{wang5}
Wang, W.,
\newblock{The $k$-Cauchy-Fueter complexes, Penrose transformation and Hartogs' phenomenon for quaternionic $k$-regular functions},
{\it J. Geom. Phys.} \textbf{60} (2010), 513-530.

\bibitem{wang2}
Wang, W.,
\newblock{On the linear algebra in the quaternionic pluripotential theory},
{\it  Linear Algebra Appl.} \textbf{562} (2019), 223-241.

\bibitem{wang4}
Wang, W.,
\newblock{Quaternionic projective invariance of the $k$-Cauchy-Fueter complex and applications I},
{\it Differ. Geom. Appl.} \textbf{101} (2025), Article 102299.


\bibitem{wulan}
Wulan, H., S.,and Zhu K.,
{\it M\"{o}bius invariant \(\mathcal Q_K\) spaces},
{\rm Springer-Verlag,  New York},  (2017).

\bibitem{Xia}
Xia, W., and Wang H. Y.,
\newblock{The Mobius addition and generalized Laplace-Beltrami operator in octonionic space},
{\it Adv. Appl. Clifford Algebr.} \textbf{34} (2024), Paper No. 27.

\bibitem{Zhang}
Zhang, G. K., and Liu H. P.,
\newblock{Realization of quaternionic discrete series on the unit ball in $\mathbb{H}^d$},
{\it J. Funct. Anal.} \textbf{262} (2012), 2979-3005.

\bibitem{Zhou}
Zhou L. F.,
\newblock{A Bohr phenomenon for $\alpha$-harmonic functions},
{\it J. Math. Anal. Appl.} \textbf{505} (2022), Ariticle 125617.
\end{thebibliography}
\end{document}